\documentclass[11pt]{amsart}
\usepackage[utf8]{inputenc}
\usepackage{amsmath}
\usepackage{amsfonts}
\usepackage{amsfonts,latexsym,rawfonts,amsmath,amssymb,amsthm}
\usepackage[plainpages=false,bookmarks=true,pdfstartview=FitH, pdfborder={0 0 0},colorlinks=true,citecolor=blue, linkcolor=blue, hypertexnames=false]{hyperref}

\usepackage{comment}

\usepackage{graphicx}
\makeatletter
\renewcommand\@biblabel[1]{}
\makeatother

\numberwithin{equation}{section}

\newcommand{\beq}{\begin{equation}}
\newcommand{\eeq}{\end{equation}}
\newcommand{\beqs}{\begin{eqnarray*}}
\newcommand{\eeqs}{\end{eqnarray*}}
\newcommand{\beqn}{\begin{eqnarray}}
\newcommand{\eeqn}{\end{eqnarray}}
\newcommand{\beqa}{\begin{array}}
\newcommand{\eeqa}{\end{array}}

\DeclareMathOperator{\AVR}{AVR}
\DeclareMathOperator{\Rm}{Rm}
\DeclareMathOperator{\Ric}{Ric}

\def\lra{\longrightarrow}

\def\bc{\begin{center}}
\def\ec{\end{center}}
\def\d#1#2{\frac{\displaystyle #1}{\displaystyle #2}}

\def\d{\delta}

\def\R{{\mathbb R}}

\def\R{{\mathfrak R}}

\def\ZZ{{\mathbb Z}}

\def\begeq{\begin{equation}}
\def\endeq{\end{equation}}
\def\and{\quad{\rm and}\quad}

\let\lra=\longrightarrow

\def\mapright\#1{\,\smash{\mathop{\lra}\limits^{\#1}}\,}

\newtheorem{prop}{Proposition}[section]
\newtheorem{theo}[prop]{Theorem}
\newtheorem{lem}[prop]{Lemma}
\newtheorem{claim}[prop]{Claim}
\newtheorem{cor}[prop]{Corollary}
\newtheorem{rem}[prop]{Remark}

\newtheorem{defi}[prop]{Definition}
\newtheorem{conj}[prop]{Conjecture}

\begin{document}

\date{}
\author {Bennett Chow}
\author {Yuxing $\text{Deng}^*$ }
\author {Zilu Ma}

\thanks {* Supported by the NSFC Grant 12022101 and 11971056.}

\address{ Bennett Chow\\ Department of Mathematics, University of California, San Diego, 9500 Gilman Drive \#0112, La Jolla, CA 92093-0112, USA\\ bechow@ucsd.edu.}

\address{ Yuxing Deng\\School of Mathematics and Statistics, Beijing Institute of Technology,
Beijing, 100081, China\\
6120180026@bit.edu.cn}

\address{ Zilu Ma\\ Department of Mathematics, University of California, San Diego, 9500 Gilman Drive \#0112, La Jolla, CA 92093-0112, USA\\ zim022@ucsd.edu.}



\title[Steady gradient Ricci $4$-solitons that dimension reduce]{On four-dimensional steady gradient Ricci solitons that dimension reduce}
\maketitle

\section*{\ }

\begin{abstract} In this paper, we will study the asymptotic geometry of 4-dimensional steady gradient Ricci solitons under the condition that they dimension reduce to $3$-manifolds. We will show that such
solitons either
strongly
dimension reduce to a spherical space form $\mathbb{S}^3/\Gamma$ or weakly dimension reduce to the $3$-dimensional Bryant soliton. We also show that
4-dimensional steady gradient Ricci
soliton \emph{singularity models} with nonnegative Ricci curvature outside a compact set either are  Ricci-flat ALE $4$-manifolds or  dimension reduce to $3$-dimensional manifolds. As a
further
application, we prove that any steady gradient K\"{a}hler-Ricci soliton singularity models on complex surfaces with  nonnegative Ricci curvature outside a compact set must be hyperk\"{a}hler ALE $4$-manifolds.
\end{abstract}

\tableofcontents

\section{Introduction}

\emph{Steady} gradient Ricci solitons arise as singularity models in the case of Type II singular solutions of the Ricci flow \cite{H2}. In dimension $4$, Ricci flow singularity formation may be quite complicated, with a
given
forming singularity possibly having different
associated singularity
models at different
curvature
scales. 

For example, Appleton's work \cite{Appleton2}
proves the existence
of $4$-dimensional Ricci flow singularity
formation
where the associated singularity model at the highest curvature scale is the Eguchi--Hanson Ricci flat ALE,
which is a steady soliton with trivial potential function,
and where
at a lower scale
the model is
either flat $\mathbb{R}^4/\mathbb{Z}_2$ or
the $\mathbb{Z}_2$-quotient of the Bryant soliton, which is a steady soliton with nontrivial potential function.

By Bamler's recent works \cite{Bam1,Bam2,Bam3}, which
solved
a version of a conjecture of Hamilton, there is a definite sense in which most singularity models in all dimensions are shrinking gradient Ricci solitons.
In particular, Bamler's theory proves that for any forming singularity one always has that at the 
appropriate (parabolic)
scale and approach to the singularity the associated singularity model is a shrinking gradient Ricci soliton or a Ricci flat cone.  

In dimension $4$, steady Ricci solitons may also be relevant to the study of \emph{shrinking} gradient Ricci solitons with quadratic curvature growth (which is the maximum possible growth) via a limit argument; see Proposition 10 in \cite{CFSZ}.
The reason why steady and shrinking gradient solitons may be related to each other in dimension 4 is that the shrinking soliton equation $\Ric - \text{Hess}  f =\frac{1}{2}g$ degenerates to $\Ric = \text{Hess}  f$ as a limit of dilations provided convergence holds, which is true in the quadratic curvature growth case.
Bamler has asked the question of whether $4$-dimensional quadratic curvature growth shrinking solitons analogous to flying wings may exist.

Throughout this paper, we will use the following notations.
A triple $(M^n,g,f)$ of a smooth manifold, a complete Riemannian metric, and a function is an $n$-dimensional steady gradient Ricci soliton, $\{\phi_t\}_{t\in(-\infty,\infty)}$ is the $1$-parameter group of diffeomorphisms generated by $-\nabla f$, and $g(t)=\phi_t^{*}g$. By definition,
the Ricci curvature  ${\rm Ric}$ of $g$ on $M$ satisfies
\begin{align}
{\rm Ric}={\rm Hess} f.
\end{align}
Defining $f(t)=f \circ \phi_t$, we have an eternal solution to the Ricci flow:
\begin{equation}
    \frac{\partial}{\partial t}g(t) = - 2 {\rm Ric}_{g(t)}= -2 {\rm Hess}_{g(t)} f(t).
\end{equation}

Regarding the qualitative study of steady gradient Ricci solitons with relatively mild conditions on curvature, save positivity, there are a number of important results, including \cite{Br1},  \cite{Br2}, \cite{Chan}, \cite{CLY11},  \cite{DZ5}, \cite{Deruelle2012}, \cite{Ham95}, \cite{MS13}, \cite{MSW19}, \cite{MW}, and \cite{Na}.

In this paper we will study the asymptotic geometry of 4-dimensional steady gradient Ricci solitons
under the assumption that they
``dimension reduce'' to
$3$-manifolds (see below for definitions). 
We also prove dimension reduction must hold only assuming that the Ricci curvature is nonnegative outside a compact set.
We first introduce the definition of dimension reduction on steady gradient Ricci solitons.

\begin{defi}\label{def-1}
We say that $(M^n,g,f)$ \textbf{dimension reduces to $(n-1)$-manifolds} if for any sequence $\{p_i\}_{i\in \mathbb{N}}$ tending to infinity,
a subsequence of
$(M,K_ig(K_i^{-1}t),p_i)$ converges to $(N^{n-1}\times \mathbb{R},g_N(t)+ds^2,p_\infty)$ in the $C^\infty$ pointed Cheeger--Gromov sense, where $(N,g_{N}(t))$, $t\in(-\infty,0]$, is an $(n-1)$-dimensional complete ancient Ricci flow with bounded curvature and where $K_i=|{\Rm}(p_i)|>0$.
In this definition, $(N,g_{N}(t))$ may depend on the choice of the base points $\{p_i\}$ and the subsequence.
We call any such $(N,g_{N}(t))$ a \textbf{dimension reduction} of $(M,g,f)$.

We say $(M^n,g,f)$ \textbf{strongly dimension reduces to} $(N^{n-1},g_{N}(t))$
provided that
$(M,g,f)$ dimension reduces to $(n-1)$-manifolds, where the dimension reduction of $(M,g,f)$ is always $(N,g_{N}(t))$ and hence is independent of the choice of $\{p_i\}$ and subsequence.
\end{defi}

\begin{defi}
We say that $(M^n,g,f)$ \textbf{weakly dimension reduces to} $(N^{n-1},g_{N}(t))$ if there exists a sequence of points $\{p_i\}_{i\in \mathbb{N}}$ tending to infinity such that $(M,K_ig(K_i^{-1}t),p_i)$ converges to $(N\times \mathbb{R},g_N(t)+ds^2,p_\infty)$, where $K_i=|{\Rm}(p_i)|>0$.
\end{defi}

Observe that in the definitions above, $(N,g_N(t))$ is always nonflat since $|{\Rm}_N|(p_\infty,0)=1$.

From now on
we assume that $(M^4,g,f)$ is $\kappa$\textbf{-noncollapsed}. By this we mean that if $|{\Rm}|\leq r^{-2}$ in $B(p,r)$, where $p\in M$ and $r\in(0,\infty)$, then $\operatorname{vol}B(p,r)\geq \kappa r^4$. If $(M^4,g)$ is a singularity model of the Ricci flow (see Definition \ref{def-singularity model}), then it satisfies the stronger property that the above holds with $|{\Rm}|$ replaced by the scalar curvature $R$ (this is a theorem of Perelman \cite[Theorems 28.6 and 28.9]{Cetc}); in this case we say that $(M^4,g)$ is \textbf{strongly $\kappa$-noncollapsed}.



According to the works of Hamilton and Perelman on ancient solutions, Brendle's remarkable solution to Perelman's conjecture on the classification of noncompact $3$-dimensional $\kappa$-solutions \cite{Br3},
and the recent classification of compact $3$-dimensional $\kappa$-solutions by Brendle, Daskalopoulos, and Sesum \cite{BDS} (see Bamler and Kleiner \cite{BamKle} for a later, alternative treatment stemming from their fundamental proof of the generalized Smale conjecture \cite{BamKle2} using the strong stability of the $3$-dimensional Ricci flow),
for any
given
sequence of points $\{p_i\}_{i\in \mathbb{N}}$, $(N^3,g_{N}(t))$ in Definition \ref{def-1}
is
one the following solutions:
\begin{enumerate}
\item quotients of $\kappa$-solutions on
${S}^3$ (such as shrinking spherical space forms $\mathbb{S}^3/\Gamma$ and Perelman's ancient solutions on $S^3$ and $\mathbb{RP}^3$);
\item the ancient Ricci flow generated by the $3$-dimensional Bryant soliton (which we henceforth abbreviate as the Bryant $3$-soliton);
\item quotients of shrinking round cylinders on $\mathbb{S}^2\times\mathbb{R}$ (the only orientable $\kappa$-noncollapsed quotient is by $\ZZ_2$
acting diagonally).
\end{enumerate}

In this paper, we
obtain
further restrictions on
the possible asymptotic geometries of  $4$-dimensional  steady gradient Ricci solitons that  dimension reduce to $3$-manifolds. Precisely, we prove the following result.

\begin{theo}\label{theorem-1}
If $(M,g,f)$ is a $4$-dimensional
steady gradient Ricci soliton that dimension reduces to $3$-manifolds, then either it
strongly
dimension reduces to $\mathbb{S}^3/\Gamma$ or it weakly dimension reduces to the
Bryant $3$-soliton.\footnote{More precisely, we mean it strongly dimension reduces to $(\mathbb{S}^3/\Gamma, g_{\mathbb{S}^3/\Gamma}(t))$, where $(\mathbb{S}^3/\Gamma, g_{\mathbb{S}^3/\Gamma}(t))$ is a group of shrinking spherical space forms.}
\end{theo}

Observe that if $(M^4,g,f)$ weakly dimension reduces to the Bryant $3$-soliton, then it also weakly dimension reduces to $\mathbb{S}^2\times\mathbb{R}$.
A key content of Theorem \ref{theorem-1} is that, under its hypothesis,
if dimension reduction to $\mathbb{S}^2\times\mathbb{R}$ occurs for a given steady soliton, then conversely dimension reduction to the Bryant $3$-soliton \emph{must} occur for that steady soliton.
Although this agrees with one's intuitive picture of the landscape of $4$-dimensional steady solitons, this is technically difficult to prove.

    Examples of steady gradient Ricci solitons that
strongly
dimension reduce to $\mathbb{S}^3/\Gamma$ are the Bryant $4$-soliton (to $\mathbb{S}^3$) and Appleton's cohomogeneity one examples on real plane bundles over $S^2$ (to $\mathbb{S}^3/\mathbb{Z}_k$, $k\geq3$); see \cite{Appleton}.
The first named author also conjectured that there exist similar steady gradient Ricci solitons on plane bundles over $\mathbb{RP}^2$. Among these solitons, only the Bryant $4$-soliton has positive Ricci curvature (it in fact has positive curvature operator).
We note that Hamilton has conjectured that there exists a family of 4-dimensional $\kappa$-noncollapsed steady gradient Ricci solitons, called \textbf{flying wings}, with positive curvature operator that weakly dimension reduce to the Bryant $3$-soliton. Lai \cite{Lai20} recently confirmed Hamilton's conjecture in dimension $3$ by proving the existence of $3$-dimensional flying wings, which reduce to the cigar soliton.


When $(M^4,g,f)$ weakly dimension reduces to the Bryant $3$-soliton and its scalar curvature has no decay, one may conjecture that $(M,g,f)$ either is the product of the Bryant $3$-soliton and a line or is a flying wing. It is also unknown whether there exists a steady gradient Ricci soliton which weakly dimension reduces to the Bryant $3$-soliton and has
scalar curvature uniformly decaying to zero.
If such a steady gradient Ricci soliton exists, we 
expect
that the asymptotic behavior of its level set flow\footnote{One may see Section 3 in \cite{DZ5} for the definition of the level set flow of a steady gradient Ricci soliton.} should be similar to that of the $3$-dimensional $\kappa$-solution constructed by Perelman.

A key idea of Theorem \ref{theorem-1} is that (\ref{positive lower bound of derivative}) holds if the steady Ricci soliton doesn't reduce to  a $3$-dimensional steady gradient Ricci soliton. Condition (\ref{positive lower bound of derivative}) actually implies that the steady Ricci soliton has linear curvature decay and hence strongly
dimension reduces to $\mathbb{S}^3/\Gamma$. Although Brendle, Daskalopoulos and Sesum \cite{BDS} have obtained the classification of $3$-dimensional compact $\kappa$-solutions, which was based on  the asymptotic behavior of $3$-dimensional compact $\kappa$-solutions studied in \cite{ABDS}, our proof does not rely on their results.


In most cases,
one is
interested in Ricci solitons which are singularity models. The definition of a singularity model is as follows.
\begin{defi}\label{def-singularity model}
Let $(M^n,g(t))$, $t\in[0,T)$, be a finite time singular solution to Ricci flow on a closed oriented manifold such that $\sup_{M\times[0,T)}|{\Rm}|=\infty$ and $T<\infty$. An associated \textbf{singularity model} $(M_{\infty}^n,g_{\infty}(t))$, $t\in(-\infty,0]$, is a complete ancient solution which is a limit of pointed rescalings. More precisely, there exists a sequence of space-time points $(x_i,t_i)$ in $M\times[0,T)$ with $K_i\triangleq |{\Rm}|(x_i,t_i)\to\infty$ such that the sequence of pointed solutions $(M,g_i(t),x_i)$, where $g_i(t)=K_ig(K_i^{-1}
t 
+t_i)$ and $t\in[-K_it_i,0]$, converges in the Cheeger-Gromov sense to the complete ancient solution $(M_{\infty},g_{\infty}(t),x_{\infty})$, $t\in (-\infty,0]$.
\end{defi}

As an application of Theorem \ref{theorem-1}, we can give a description of  the asymptotic behavior of any $4$-dimensional steady gradient Ricci soliton singularity model $(M,g,f)$ whose Ricci curvature is nonnegative outside a compact set $K$, i.e., its Ricci curvature ${\rm Ric}(x)$ satisfies
 \begin{align}\label{Ricci nonnegative outside K}
 {\rm Ric}(x)\ge0,\quad\forall~x\in M\setminus K.
 \end{align}
 Here, dimension reduction is not one of the hypotheses.

\begin{theo}\label{theo-singularity model}
Let $(M,g,f)$ be a $4$-dimensional steady gradient Ricci soliton singularity model satisfying \eqref{Ricci nonnegative outside K}. Then, one of the following holds:

(1) $(M,g,f)$ is a Ricci flat ALE $4$-manifold;

(2) $(M,g,f)$ strongly dimension reduces to $\mathbb{S}^3/\Gamma$;

(3) $(M,g,f)$ weakly dimension reduces to the Bryant $3$-soliton.
\end{theo}

The aforementioned Appleton's cohomogeneity-one steady soliton examples have positive curvature operator outside a compact set. Hamilton's conjectured flying wings are expected to have at least positive Ricci curvature outside a compact set.
At the present time, all known (non-splitting) noncollapsed $4$-dimensional steady solitons have positive curvature outside a compact set.

ALE $4$-manifolds are defined as follows.

\begin{defi}
A complete noncompact Riemannian $4$-manifold $(M^4, g)$ is an \textbf{asymptotically
locally Euclidean (ALE) space of order $\tau > 0$} if there exist a finite subgroup
$\Gamma$ of $SO(4)$, a compact subset $K$ of $M$, and a diffeomorphism
\begin{align}
\Phi:(\mathbb{R}^4-B_1( \textbf{0}))/\Gamma\to M\setminus K
\end{align}
such that $\widetilde{g}_{ij}=\pi^{\ast}\Phi^{\ast}g$, where $\pi:\mathbb{R}^4-B_1(\textbf{0})\to (\mathbb{R}^4-B_1(\textbf{0}))/\Gamma$ is the projection, which satisfies $|\widetilde{g}_{ij}-\delta_{ij}|\le O(r^{-\tau})$ and $|\partial^{\textbf{k}}g_{ij}|\le O(r^{-\tau-|\textbf{k}\,|})$ for multi-indices $\textbf{k}$.
\end{defi}

When the steady Ricci solitons in Theorem \ref{theo-singularity model} are K\"{a}hler-Ricci solitons, we can \emph{classify} any steady gradient K\"{a}hler-Ricci soliton singularity model $(M,g,f)$ of complex dimension $2$, whose Ricci curvature is nonnegative outside a compact set. Precisely, we have:

\begin{theo}\label{theo-Kaehler-2}
Steady gradient K\"{a}hler-Ricci soliton singularity models on complex surfaces must be hyperk\"{a}hler ALE Ricci-flat $4$-manifolds if they satisfy condition \eqref{Ricci nonnegative outside K}.
\end{theo}

By Bando, Kasue, and Nakajima \cite{BKN}, for any $4$-dimensional ALE, there exists $\Phi$ so
that the order of the ALE is $4$. It is conjectured that \textit{any ALE Ricci
flat 4-manifold must be hyperk\"{a}hler}. The conjecture is true when the $4$-manifold is a K\"{a}hler manifold  by  \cite{Kr1,Kr2}.
 Simply-connected
hyperk\"{a}hler ALE 4-manifolds have been classified by Kronheimer \cite{Kr1,Kr2}.  The non-simply-connected hyperk\"{a}hler ALE $4$-manifolds  have also been classified by Suvaina \cite{Su} and Wright \cite{Wr}. We
remark  that K\"{a}hler-Ricci
flat non-ALE spaces  may have infinite type; see \cite{AKL}.
Note also that Appleton's cohomogeneity-one examples are non-K\"{a}hler.

 Now, we explain the ideas in  proving Theorem \ref{theo-singularity model} and Theorem \ref{theo-Kaehler-2}. When the scalar curvature has no uniform decay, we will prove a dimension reduction theorem for steady Ricci solitons. In \cite{DZ4}, Xiaohua Zhu and the second named author proved the dimension reduction for steady gradient K\"{a}hler-Ricci solitons with nonnegative bisectional curvature. Under (\ref{Ricci nonnegative outside K}), we can find a geodesic line in the limit of a sequence of steady gradient Ricci solitons.
 In this paper we will prove the following dimension reduction theorem.

\begin{theo}\label{theo-dimension reduce without decay}
 Let $(M,g,f)$ be an $n$-dimensional
 $\kappa$-noncollapsed
 steady gradient Ricci soliton with bounded curvature. If $(M,g,f)$ satisfies \eqref{Ricci nonnegative outside K} and does not have uniform scalar curvature decay, then it weakly dimension reduces to an  $(n-1)$-dimensional steady gradient Ricci soliton.
 \end{theo}

Recall that in dimension $4$, bounded curvature holds in the case that the steady gradient Ricci soliton is a singularity model.

By Theorem \ref{theo-dimension reduce without decay}, to prove Theorem \ref{theo-singularity model} it suffices to deal with $4$-dimensional  $\kappa$-noncollapsed steady gradient Ricci solitons satisfying (\ref{Ricci nonnegative outside K}) and uniform scalar curvature decay, i.e., $R(x)\rightarrow0$ as $x\rightarrow\infty$.
We prove the following.

\begin{theo}\label{theo-nonnegative Ricci}
Let $(M,g,f)$ be a $4$-dimensional $\kappa$-noncollapsed steady gradient Ricci soliton with uniform scalar curvature decay. If it satisfies condition \eqref{Ricci nonnegative outside K}, then one of the following holds:

(1) $(M,g,f)$ is Ricci flat; 

(2) $(M,g,f)$ strongly dimension reduces to $\mathbb{S}^3/\Gamma$;

(3) $(M,g,f)$ weakly dimension reduces to the Bryant $3$-soliton.
\end{theo}

It turns out that it is important to study steady gradient Ricci solitons with maximal volume growth in order to prove Theorem \ref{theo-nonnegative Ricci}. We say that an $n$-dimensional steady gradient Ricci soliton $(M,g,f)$ satisfying condition (\ref{Ricci nonnegative outside K}) has \textbf{maximal volume growth}
if the \textbf{asymptotic volume ratio (AVR)}
\[
	\AVR(g) := \lim_{r\to \infty} \frac{V_x(r)}{r^n} >0,
\]
where $V_x(r) = {\rm vol}\, B(x,r)$. We will show in Lemma \ref{AVR well defined} in the appendix that the limit exists and does not depend on the basepoint $x$. For such steady gradient Ricci solitons with uniform curvature decay,
we will prove in Proposition \ref{uniform vol ratio} that if $\AVR(g)>0$, then there is
a uniform constant $c>0$ depending only
on the Ricci lower bound over $K$ 
such that
\[
	\frac{V_y(r)}{r^n} \ge c\AVR(g),\quad \forall y\in M\setminus K, \; r>0.
\]

For $4$-dimensional steady gradient Ricci solitons with maximal volume growth, we have the following rigidity theorem.

\begin{theo}\label{theo-maximal volume growth}
Let $(M,g,f)$ be a $4$-dimensional steady gradient Ricci soliton
with uniform scalar curvature decay.
If $(M,g,f)$ satisfies condition \eqref{Ricci nonnegative outside K} and has maximal volume growth, then $(M,g)$ must be Ricci flat.
\end{theo}

 It is interesting to know whether $n$-dimensional steady gradient Ricci solitons should be Ricci flat if their volume growth satisfies
 \begin{align}
     {\rm vol}\,B(x_0,r_i)\ge cr_i^n,\notag
      \end{align}
 where $c$ is a positive constant and  $x_0$ is a fixed point on $M$. Moreover, $r_i$ is a sequence of positive constants such that $r_i\to\infty$ as $i\to\infty$.

With the help of  Theorem \ref{theo-maximal volume growth}, we are able to show the following convergence result.
\begin{theo}\label{theo-compactness}
Let $(M^4,g,f)$ be a steady gradient Ricci soliton which is not Ricci flat and satisfies the hypotheses of Theorem \ref{theo-nonnegative Ricci}. Then, for any $p_i$ tending to infinity, $\big(M, R(p_i)g(R(p_i)^{-1}t),p_i\big)$ converges subsequentially to a complete limit $(M_{\infty},g_{\infty}(t),p_{\infty})$. Moreover, $(M_{\infty},g_{\infty}(t),p_{\infty})$ has uniformly bounded curvature and  splits off a line.
\end{theo}

Theorem \ref{theo-nonnegative Ricci} is then a corollary of Theorem \ref{theorem-1} and Theorem \ref{theo-compactness}.  Combining Theorem \ref{theo-nonnegative Ricci} with Theorem \ref{theo-dimension reduce without decay}, we obtain:
\begin{theo}\label{theo-nonnegative Ricci-bounded curvature}
Let $(M,g,f)$ be a $4$-dimensional $\kappa$-noncollapsed steady gradient Ricci soliton with bounded curvature. If it satisfies \eqref{Ricci nonnegative outside K}, then one of the following holds:

(1) $(M,g,f)$ is Ricci flat;

(2) $(M,g,f)$ strongly dimension reduces to $\mathbb{S}^3/\Gamma$;

(3) $(M,g,f)$ weakly dimension reduces to the Bryant $3$-soliton.
\end{theo}

When the steady gradient Ricci solitons in Theorem \ref{theo-nonnegative Ricci-bounded curvature}  are K\"{a}hler-Ricci solitons, we have:

\begin{theo}\label{theo-Kaehler-1}  
$\kappa$-noncollapsed steady gradient K\"{a}hler-Ricci solitons with bounded curvature on complex surfaces must be Ricci flat if they satisfy \eqref{Ricci nonnegative outside K}.
\end{theo}

Theorem \ref{theo-singularity model} is in fact a direct corollary of Theorem \ref{theo-nonnegative Ricci-bounded curvature} by Theorem 1 in \cite{CFSZ}, Perelman's no local collapsing theorem and the work of Cheeger and Naber \cite[Corollary 8.86]{CN}. Similarly, Theorem \ref{theo-Kaehler-2} is a direct corollary of Theorem \ref{theo-Kaehler-1}.

\begin{conj}
If $(M,g,f)$ is a $4$-dimensional steady gradient Ricci soliton singularity model, then it dimension reduces to $3$-manifolds.
\end{conj}

Regarding Ricci flow analysis, compactness theory, and singularity models, particularly striking are the recent breakthroughs of Bamler \cite{Bam1,Bam2,Bam3}.
In \cite{BCDMZ} by Bamler, Zhang, and the three authors, Bamler's theory is applied to obtain results regarding the tangent flows at infinity of $4$-dimensional steady soliton singularity models.

This paper is organized as follows. In Section \ref{section 3}, we study the linear growth of the potential function. In Section \ref{section-linear decay}, we study the linear curvature decay of steady gradient Ricci solitons. Sections \ref{section-AG}, \ref{section-Uniqueness of limit}, \ref{section-limit soliton is Bryant} are devoted to proving Theorem \ref{theorem-1}. Theorem \ref{theo-maximal volume growth} is proved in Section \ref{section-volume growth}. In Section \ref{section-compactness theorem on solitons}, we prove Theorem \ref{theo-compactness}. Theorem \ref{theo-dimension reduce without decay} is proved in Section \ref{section-dimension reduction}. Theorem \ref{theo-singularity model}, Theorem \ref{theo-Kaehler-1} and Theorem \ref{theo-Kaehler-2} will be proved in Section \ref{final section}.\smallskip

\noindent\textbf{Acknowledgment.} The third-named author would like to thank Professors Lei Ni and Peng Lu for helpful discussions.
We would like to thank the anonymous referee for many helpful suggestions which improved the paper, including pointing out a mistake in one of the proofs. 


\section{Linear growth of the potential function}\label{section 3}
The growth of the potential function $f$ is important for studying the rotational symmetry of steady gradient Ricci solitons (see \cite{Br1,Br2}, \cite{DZ5}). It is
known
that $f$ grows linearly when the Ricci curvature is nonnegative and there exists an equilibrium point $o$ of $f$ on $M$, i.e., $\nabla f(o)=0$ (see \cite{CaCh}, \cite{CaNi}). Let $\phi_t$ be the $1$-parameter group of diffeomorphisms generated by $-\nabla f$. Then the nonnegativity of the Ricci curvature implies that $\frac{\rm d}{{\rm d} t}R(\phi_t(p))\ge 0$. In this section we will show that $f$ grows linearly under the assumption that $\frac{\rm d}{{\rm d} t}R(\phi_t(p))\ge 0$ outside some compact set $K$ and that the scalar curvature decays uniformly. Precisely, we have:

\begin{theo}\label{theo-linear of f}
Let $(M^n,g,f)$ be a non-Ricci-flat steady gradient Ricci soliton. Suppose that the scalar curvature decays uniformly and
\begin{align}\label{eq: dRdt geq 0 on M minus K}
\left.\frac{\rm d}{{\rm d} t}\right|_{t=0}R(\phi_t(p))\ge 0 ~\text{ for all }p\in M\setminus K,
\end{align}
where $K$ is a compact subset of $M$.
Then there exist
positive
constants $r_0$, $C_1$ and $C_2$ such that
\begin{align}
C_1\rho(x)\le f(x)\le C_2\rho(x) ~\text{ for all }x\in M\text{ such that }\rho(x)\ge r_0,
\end{align}
where $\rho(x)$ is the distance function from a fixed point $x_0\in M$.
That is, the potential function is uniformly equivalent to the distance (to a fixed point) function.
\end{theo}

\textbf{Remark.}
Condition \eqref{eq: dRdt geq 0 on M minus K} is equivalent to $\langle \nabla R,\nabla f \rangle \leq 0$ on $M\setminus K$, which in turn is equivalent to $\Ric(\nabla f,\nabla f)\geq 0$ on $M\setminus K$.\smallskip

Now we fix some notations in this section.
Let $(M^n,g,f)$ be a non-Ricci-flat steady gradient Ricci soliton with scalar curvature decaying uniformly.
It is well known that the following identity holds:
\begin{align}\label{identity}
R(x)+|\nabla f|^2(x)=C,
\end{align}
where $C$ is a positive constant.
Since
$(M,g,f)$ is non-Ricci-flat, $R(x)$ must be positive by B.-L.~Chen \cite{Ch}. In particular, \eqref{identity} implies that
\begin{align}
|\nabla f|^2(x)\le C~\text{ for all }x\in M.
\end{align}
Let $R_{\max}=\sup_{x\in M} R(x)$ and define
$$S(\varepsilon)=\{x\in M : R(x)\ge R_{\max}-\varepsilon\}.$$
Note that $R_{\max}\le C$ and $S(\varepsilon)=\{x\in M : |\nabla f|^2\le \varepsilon+C-R_{\max}\}$,  where $C$ is the constant in (\ref{identity}).
Since the scalar curvature decays uniformly, $S(\varepsilon)$ is compact for each $\varepsilon\in [0,R_{\max})$. Moreover, $M$ is exhausted by
the family of sets
$\{S(\varepsilon) \}_{\varepsilon\in (0,R_{\max}]}$. Hence, there exists a positive constant $\varepsilon_0<R_{\max}$ such that
\begin{align}
K\subseteq S(\varepsilon) ~\text{ for all }~\varepsilon\in [\varepsilon_0,R_{\max}),
\end{align}
where $K$ is as in the hypothesis of Theorem \ref{theo-linear of f}.

\begin{lem}\label{lem-remain for t negative}
Under the hypotheses of Theorem \ref{theo-linear of f}, for any $p\in M\setminus S(\varepsilon_0)$, we have
\begin{align}
\phi_{t}(p)\in M\setminus S(\varepsilon_0)~\text{ for all }t \in (-\infty,0]. 
\end{align}
Hence,
\begin{align}
\frac{\rm d}{{\rm d} t}R(\phi_t(p))\ge 0 ~\text{ for all }(p,t)\in \left( M\setminus S(\varepsilon_0) \right) \times(-\infty,0].
\end{align}
\end{lem}
\begin{proof}
Let $p\in M\setminus S(\varepsilon_0)$.
By the definition of $S(\varepsilon_0)$, we have that $|\nabla f|^2(p)>\varepsilon_0+C-R_{\max}$. Therefore, for
$t_0<0$ sufficiently small, we have that $|\nabla f|^2(\phi_t(p))>\varepsilon_0+C-R_{\max}$ for all $t\in[t_0,0]$. Let
$$T=\inf\{t\le0 :\phi_t(p)\in  M\setminus S(\varepsilon_0)\}.$$
If $T$ is finite, then $|\nabla f|^2(\phi_t(p))
>
\varepsilon_0+C-R_{\max}$ for $T<t\le0$ and $|\nabla f|^2(\phi_T(p))=\varepsilon_0+C-R_{\max}$. Note that
\begin{align}
\frac{{\rm d}}{{\rm d}t}|\nabla f|^2(\phi_{t}(p))=-\frac{{\rm d}}{{\rm d}t}R(\phi_{t}(p))\le 0 ~\text{ for all }~(p,t)\in ( M\setminus S(\varepsilon_0) ) \times(T,0]. \notag
\end{align}
Hence,
\begin{align}
|\nabla f|^2(\phi_T(p))\ge |\nabla f|^2(\phi_t(p))>\varepsilon_0+C-R_{\max}
\end{align}
for $t\in(T,0]$.
This
contradicts the fact that $\phi_T(p)\in S(\varepsilon_0)$. Hence $T=-\infty$, i.e., $\phi_t(p)\in M\setminus S(\varepsilon_0)$ for all $t\le0$.
We
have
completed the proof.
\end{proof}
 When $t\to+\infty$, we
have the following lemma.

\begin{lem}\label{lem-shrinking}
Under the hypotheses of Theorem \ref{theo-linear of f}, for any $p\in M\setminus S(\varepsilon_0)$, there exists
$t_p\in(0,\infty)$
such that $\phi_{t_p}(p)\in S(\varepsilon_0)$ and $\phi_{t}(p)\in M\setminus S(\varepsilon_0)$ for all
$t\in(-\infty,t_p)$.
\end{lem}
\begin{proof}
The proof is by contradiction. If the lemma is not true, then there exists a point $p\in M\setminus S(\varepsilon_0)$ such that $\phi_t(p)\in M\setminus S(\varepsilon_0)$ for all $t\ge0$. Then
\begin{align}
|\nabla f|^2(\phi_t(p))\ge \varepsilon_0~\text{ for all }~t\ge 0.
\end{align}
It follows that
\begin{align}\label{lower estimate of difference}
f(p)-f(\phi_t(p))=\int_{0}^t|\nabla f|^2(\phi_s(p))ds\ge \varepsilon_0t.
\end{align}
For any fixed $t\ge 0$, let $\gamma_t:[0,L_t]\to M$ be a minimal geodesic with $\gamma_t(0)=p$ and $\gamma_t(L_t)=\phi_t(p)$, where $L_t=d(p,\phi_t(p))$, and let $s$ be the arc length parameter. Then we have
\begin{align}\label{upper estimate of difference}
|f(p)-f(\phi_t(p))|=&\left|\int_{0}^{L_t}\langle\nabla f(\gamma_t(s)),\gamma_t^{\prime}(s)\rangle ds\right|\\
\le& \int_{0}^{L_t}|\nabla f|ds\notag\\
\le& \sqrt{C}d(p,\phi_t(p)), \notag
\end{align}
where we have used the identity (\ref{identity}). Combining (\ref{lower estimate of difference}) with (\ref{upper estimate of difference}), we see that $\phi_t(p)$ tends to infinity as $t\to+\infty$.

On the other hand,  $\frac{\rm d}{{\rm d} t}R(\phi_t(p))\ge 0$ since $\phi_t(p)\in M\setminus S(\varepsilon_0)$ for all $t\ge0$. It follows that $\phi_t(p)\in S(R_{\max}-R(p))$ for $t\ge0$. Note that $S(R_{\max}-R(p))$ is compact. Hence, $\phi_t(p)$ cannot tend to infinity. We obtain a contradiction. Hence the proof of the lemma is complete.

\end{proof}

Now, we are ready to prove Theorem \ref{theo-linear of f}.

\begin{proof}[Proof of Theorem \ref{theo-linear of f}]
For any $x\in M\setminus S(\varepsilon_0)$, by Lemma \ref{lem-shrinking} there exists a positive constant $t_x>0$ such that $\phi_{t_x}(x)\in S(\varepsilon_0)$ and $\phi_{t}(x)\in M\setminus S(\varepsilon_0)$ for all
$t\in(-\infty,t_x)$.
Note that
\begin{align}
f(x)-f(\phi_{t_x}(x))=\int_{0}^{t_x}|\nabla f|^2(\phi_{s}(x))\, {\rm d}s \notag
\end{align}
and
\begin{align}
d(x,\phi_{t_x}(x))\le{\rm Length}(\phi_s(x)|_{s\in[0,t_x]},g)=\int_{0}^{t_x}|\nabla f|(\phi_{s}(x)) {\rm d}s.\notag
\end{align}
By (\ref{identity}), we have
\begin{align}
\frac{{\rm d}}{{\rm d}t}|\nabla f|^2(\phi_{t}(x))=-\frac{{\rm d}}{{\rm d}t}R(\phi_{t}(x))\le 0\,\text{ for all }t\in [0,t_x], \notag
\end{align}
\begin{align}
|\nabla f|(\phi_{s}(x))\ge |\nabla f|(\phi_{t_x}(x))\,\text{ for all }s\in[0,t_x].\notag
\end{align}
Consequently,
\begin{align}
f(x)-f(\phi_{t_x}(x))=&\int_{0}^{t_x}|\nabla f|^2(\phi_{s}(x)) {\rm d}s \notag\\
\ge\,& |\nabla f|(\phi_{t_x}(x))\int_{0}^{t_x}|\nabla f|(\phi_{s}(x)) {\rm d}s  \notag\\
\ge\,& \varepsilon_0d(x,\phi_{t_x}(x)).\notag
\end{align}
Now we fix a point $x_0\in S(\varepsilon_0)$.
We then obtain for all $x\in M\setminus S(\varepsilon_0)$,
\begin{align}\label{f-lower estimate}
f(x)-f(x_0) & \ge \varepsilon_0d(x,\phi_{t_x}(x)) + f(\phi_{t_x}(x)) - f(x_0)  \\
& \ge
\varepsilon_0 d(x,x_0)- (2+\varepsilon_0)A , \nonumber
\end{align}
where $$A=\sup_{x\in S(\varepsilon_0)}|f(x)|+\operatorname{diam}(S(\varepsilon_0))$$
and $\operatorname{diam}(S(\varepsilon_0))=\sup_{x,y\in S(\varepsilon_0)}d(x,y)$
.

On the other hand, for any $x\in M$ we have the following. Let $\gamma:[0,L]\to M$ be a minimal geodesic with $\gamma(0)=x_0$ and $\gamma(L)=x$, where $L=d(x,x_0)$, and let $s$ be the arc length parameter. Then we have
\begin{align}\label{f-upper estimate}
|f(x)-f(x_0)|=
\left|
\int_{0}^{L}\langle\nabla f(\gamma(s)),\gamma^{\prime}(s)\rangle ds
\right|
\le \int_{0}^{L}|\nabla f|ds
\le \sqrt{C}d(x_0,x).
\end{align}
Combining (\ref{f-lower estimate}) with (\ref{f-upper estimate}) completes the proof of Theorem \ref{theo-linear of f}.
\end{proof}

\begin{rem}
\label{rem: linear f}
Under the hypotheses of Theorem \ref{theo-linear of f}, it is easy to see that the constant $C$ in \eqref{identity} satisfies $C=R_{\max}=\sup_{x\in M}R(x)$. Moreover, the linear growth estimate of $f$ can be improved to:
\begin{align}\label{limit constant}
 \frac{f(x)}{\rho(x)}\to \sqrt{R_{\max}} ~\text{ as }\rho(x)\to\infty.
\end{align}
The details can be found in the proof of Lemma 2.2 in \cite{DZ6}.
\end{rem}

\section{Linear curvature decay of steady GRS}\label{section-linear decay}

Linear curvature decay is an important condition in the study of the asymptotic geometry of steady gradient Ricci solitons (see \cite{Br1,Br2}, \cite{DZ2,DZ3,DZ4,DZ5,DZ6}). Originally, linear curvature decay of steady gradient Ricci solitons were obtained on positively curved steady gradient Ricci solitons in dimension $3$ by Guo (see \cite{Guo}).

In this section, we prove linear curvature decay under
conditions stronger than
that of Theorem \ref{theo-linear of f}. Let $\rho(x)$ denote the distance function from a fixed point $x_0\in M$. Let $S(\varepsilon_0)$ be the set defined as in Section \ref{section 3}.

\begin{theo}\label{theo-decay}
Let $(M^n,g,f)$ be a non-Ricci-flat steady gradient Ricci soliton. Suppose that the scalar curvature decays uniformly and that
\begin{align}\label{positive lower bound of derivative}
\frac{1}{R^2(p)}\cdot\left.\frac{\rm d}{{\rm d} t}\right|_{t=0 }R(\phi_t(p))\ge \epsilon>0\,\text{ for all }\,p\in M\setminus K,
\end{align}
where $K$ is a compact subset of $M$,
$\phi_t$ is the $1$-parameter group of diffeomorphisms generated by $-\nabla f$,
and $\epsilon$ is independent of $p,t$. Then there exist constants $r_0$ and $c$ such that
\begin{align}\label{linear decay}
 R(x)\le \frac{c}{\rho(x)}\,\text{ for all } x\in M \text{ such that }\rho(x)\ge r_0.
\end{align}
That is, the scalar curvature decays linearly.
\end{theo}

\begin{proof}
Note that $(M,g,f)$ satisfies the hypotheses of Theorem \ref{theo-linear of f}.
By Lemma \ref{lem-remain for t negative}, we see that
\begin{align}\label{inequality for t negative-1}
\frac{1}{R^2(p)}\cdot\frac{\rm d}{{\rm d} t}R(\phi_t(p))\ge \epsilon>0\,\text{ for all }\,p\in (M\setminus S(\varepsilon_0))\times(-\infty,0].
\end{align}
By Lemma \ref{lem-shrinking}, for any $x\in M\setminus S(\varepsilon_0)$ there exists
$t_x>0$ such that $\phi_{t_x}(x)\in S(\varepsilon_0)$ and $\phi_{t}(x)\in M\setminus S(\varepsilon_0)$ for all
$t\in(-\infty,t_x)$.
By (\ref{inequality for t negative-1}), we have
\begin{align}
-\frac{\rm d}{{\rm d}t}\big[R^{-1}(\phi_t(x))\big]\ge \epsilon\,\text{ for all }\,t\in[0,t_x].
\end{align}
Integrating this formula over $t\in[0,t_x]$, we obtain
\begin{align}\label{upper-1}
R(x)\le\frac{1}{\epsilon t_x+R^{-1}(\phi_{t_x}(x))}.
\end{align}
On the other hand,
\begin{align}\label{upper-2}
f(x)-f(\phi_{t_x}(x))=\int_{0}^{t_x}|\nabla f|^2(\phi_{s}(x))\, {\rm d}s\le Ct_x.
\end{align}
By (\ref{upper-1}) and (\ref{upper-2}), we have
\begin{align}\label{decay according f}
R(x)\le \frac{C}{\epsilon(f(x)-A)+C(R_{\max})^{-1}},
\end{align}
where $A=\sup_{x\in S(\varepsilon_0)}f(x)$. Hence the curvature decay estimate (\ref{linear decay}) follows from (\ref{decay according f}) and Theorem \ref{theo-linear of f}.
\end{proof}

Similarly, we
have the same linear decay estimate
under an upper bound rather than a lower bound for $-\frac{\rm d}{{\rm d}t}(R^{-1}\circ\phi_t)$.

\begin{theo}\label{theo-decay lower bound}
Let $(M^n,g,f)$ be a non-Ricci-flat steady gradient Ricci soliton. Suppose that the scalar curvature decays uniformly and
\begin{align}
0\le\frac{1}{R^2(p)}\cdot\left.\frac{\rm d}{{\rm d} t}\right|_{t=0}R(\phi_t(p))\le C^{\prime}\,\text{ for all }\,p\in M\setminus K,
\end{align}
where $K$ is a compact subset of $M$
and $\epsilon$ is independent of $p,t$. Then there exist
positive
constants $r_0$ and $c^{\prime}$ such that
\begin{align}\label{lower linear estimate}
 R(x)\ge \frac{c^{\prime}}{\rho(x)}\,\text{ for all }\rho(x)\ge r_0.
\end{align}
\end{theo}
\begin{proof}
Note that $(M,g,f)$ satisfies the hypotheses of Theorem \ref{theo-linear of f}.
By Lemma \ref{lem-remain for t negative}, we have
\begin{align}\label{inequality for t negative-2}
0\le\frac{1}{R^2(p)}\cdot\frac{\rm d}{{\rm d} t}R(\phi_t(p))\le C^{\prime}\,\text{ for all }\,(p,t)\in (M\setminus S(\varepsilon_0))\times(-\infty,0].
\end{align}
Let $x\in M\setminus S(\varepsilon_0)$.
By Lemma \ref{lem-shrinking}, there exists a constant $t_x>0$ such that $\phi_{t_x}(x)\in S(\varepsilon_0)$ and $\phi_{t}(x)\in M\setminus S(\varepsilon_0)$ for all $t\in(-\infty,t_x)$. By (\ref{inequality for t negative-2}), we have
\begin{align}
-\frac{\rm d}{{\rm d}t}\big[R^{-1}(\phi_t(x))\big]\le C^{\prime}\,\text{ for all }\,t\in[0,t_x].
\end{align}
Integrating the formula above over $t\in[0,t_x]$, we obtain
\begin{align}\label{upper-1-0}
R(x)\ge\frac{1}{C^{\prime} t_x+R^{-1}(\phi_{t_x}(x))}.
\end{align}
On the other hand,
\begin{align}\label{upper-2-0}
f(x)-f(\phi_{t_x}(x))=\int_{0}^{t_x}|\nabla f|^2(\phi_{s}(x)) {\rm d}s\ge \varepsilon_0t_x.
\end{align}
By (\ref{upper-1-0}) and (\ref{upper-2-0}), we have
\begin{align}\label{lower linear estimate according f}
R(x)\ge \frac{\varepsilon_0}{C^{\prime}f(x)+\varepsilon_0B},
\end{align}
where $B=\sup_{x\in S(\varepsilon_0)}R^{-1}(x)$. Hence (\ref{lower linear estimate}) follows from (\ref{lower linear estimate according f}) and (\ref{theo-linear of f}).
\end{proof}

\section{Asymptotic geometry}\label{section-AG}

By the
methods
in \cite{DZ5,DZ6}, one can prove Theorem \ref{theorem-1} with the help of Theorem \ref{theo-decay} and Theorem \ref{theo-decay lower bound}.

\begin{lem}\label{lem-estimates from asymptotical geometry}
Let $(M,g,f)$ be a $4$-dimensional
steady gradient Ricci soliton that dimension reduces to $3$-manifolds.
Then there
exists a positive constant $C_3$
such that
\begin{align}\label{inequality-1}
\frac{|{\Rm}|(x)}{R(x)}\le C_3\,\text{ for all }\,x\in M^4\setminus K.
\end{align}
If no dimension reduction of $(M^4,g,f)$ is
a
steady gradient Ricci $3$-soliton, then there
exist positive constants $\epsilon$ and $C_4$
such that
\begin{align}\label{inequality-2}
\epsilon\le \frac{\Delta R(x)+2|{\Ric}|^2(x)}{R^2(x)}\le C_4\,\text{ for all }\,x\in M\setminus K.
\end{align}
\end{lem}

\begin{proof}
We first show that (\ref{inequality-1}) is true. If it is not true, then there exists a sequence of points $\{p_i\}$ tending to infinity such that $\frac{R(p_i)}{|{\Rm}|(p_i)}\to0$. Let $K_i=|{\Rm}|(p_i)$.
Then $(M,K_i g(K_i^{-1}t),p_i)$ converges to $(M_{\infty}^4,g_{\infty}(t),p_{\infty})$,
where
$(M_{\infty},g_{\infty}(t))$ is the product of a $3$-dimensional ancient Ricci flow and a line. By Chen \cite{Ch}, $(M_{\infty},g_{\infty}(t))$
has
nonnegative sectional curvature.
Since
$\frac{R(p_i)}{|{\Rm}|(p_i)}\to0$,
the scalar curvature of $(M_{\infty},g_{\infty}(t))$ is zero for all $t$. Hence, $(M_{\infty},g_{\infty}(t))$ must be flat. However,  we have $|{\Rm}_{g_{\infty}(0)}|(p_{\infty})=1$ by the definition of the sequence. Hence, we obtain a contradiction. Thus (\ref{inequality-1}) is proved.

Next, we show that the
second inequality in
(\ref{inequality-2}) is true. By the convergence assumption, it is easy to see by a contradiction argument that
\begin{align}\label{inequality-3}
\frac{\Delta R(x)+2|{\Ric}|^2(x)}{|{\Rm}|^2(x)}\le C_4\,\text{ for all }x\in M\setminus K.
\end{align}
Then, the inequality on the right-hand side of (\ref{inequality-2}) follows from (\ref{inequality-1}) and (\ref{inequality-3}).

Finally,
we prove the
first inequality in
(\ref{inequality-2}). If the inequality is not true, then there exists a sequence of points $p_i$ tending to infinity such that
\begin{align}\label{limit-1}
\frac{\Delta R(p_i)+2|{\Ric}|^2(p_i)}{R^2(p_i)}\to 0~{\rm as}~i\to\infty
\end{align}
(by (\ref{inequality-1}), $|{\Rm}|$ and $R$ are comparable).
By hypothesis, we may assume that
\begin{align}
(M,R(p_i)g(R^{-1}(p_i)t),p_i)\to(M_{\infty}^4,g_{\infty}(t),p_{\infty}).
\end{align}
Also, by hypothesis,
we
have
that $(M_{\infty},g_{\infty}(t))$ is a product of a line and a complete ancient Ricci flow with bounded
nonnegative curvature operator. We also note that (\ref{limit-1}) implies that
\begin{align}
\left.\frac{\partial }{\partial t}\right|_{t=0}R_{g_{\infty}(t)}(p_{\infty})=\frac{\Delta_{g_{\infty}(0)} R_{g_{\infty}(0)}(p_{\infty})+2|{\Ric}_{g_{\infty}(0)}|_{g_{\infty}(0)}^2(p_{\infty})}{R^2_{g_{\infty}(0)}(p_{\infty})}=0.
\end{align}
Hence, by Hamilton's eternal solutions result \cite{H2}, we have that $(M_{\infty},g_{\infty}(t))$ is a product of a line and a steady gradient Ricci $3$-soliton.\footnote{Hamilton's eternal solutions result holds for complete \emph{ancient} solutions to the Ricci flow with bounded nonnegative curvature operator and with $\frac{\partial R}{\partial t}=0$ at a point.}
This contradicts our assumption. Hence, we have completed the proof of (\ref{inequality-2}).

\end{proof}

As a corollary, we have:

\begin{cor}\label{cor-non decay case}
Let $(M,g,f)$ be a $4$-dimensional
steady gradient Ricci soliton that dimension reduces to $3$-manifolds.
If the scalar curvature does not have uniform decay, then $(M^4,g,f)$ weakly dimension reduces to a steady gradient Ricci $3$-soliton.
\end{cor}

\begin{proof}
 Let $A=\lim_{r\to\infty}\sup_{x\in M\setminus B(x_0,r)}R(x)$, where $x_0$ is a fixed point. If the scalar curvature does not have uniform decay, then $A>0$. We can choose a sequence of points $\{p_i\}$ tending to infinity such that $R(p_i)\to A$ as $i\to \infty$. By Definition \ref{def-1} and by (\ref{inequality-1}), there exists a constant $C(A)$ such that
 \begin{align*}
     |{\rm Rm}|_{g_{i}(0)}(x,t)\le C(A)\quad\forall~x\in M,~t\in(-\infty,+\infty),
 \end{align*}
 where $g_{i}(t)=R(p_i)g(R^{-1}(p_i)t)$. Hence, $(M^4,g_i(t),p_i)$
subconverges
 to a limit $(N^3\times\mathbb{R},g_{N}(t)+ds^2,p_{\infty})$, where  $g_{N}(t)$ is an eternal flow.   By our choice of the sequence $\{p_i\}$, we have that $$R_{N}(p_{\infty},0)=1=\sup_{(x,t)\in N\times(-\infty,+\infty)}R_N(x,t),$$
 i.e., $R_N(x,t)$ attains its maximum in space-time at the space-time point $(p_{\infty},0)$. Hence, $(N,g_N(t))$ must admit a steady gradient Ricci soliton structure (see \cite{H2}).
\end{proof}

The steady gradient Ricci $3$-soliton in the corollary above must in fact be the Bryant $3$-soliton; see Theorem \ref{theo-uniqueness of limit soliton} below.

Note that
\begin{align}
\frac{1}{R^2(p)}\cdot\left.\frac{\rm d}{{\rm d} t}\right|_{t=0}R(\phi_t(p))=\frac{\Delta R(p)+2|{\Ric}|^2(p)}{R^2(p)}.
\end{align}

By Lemma \ref{lem-estimates from asymptotical geometry}, Corollary \ref{cor-non decay case}, Theorem \ref{theo-linear of f}, Theorem \ref{theo-decay}
 and Theorem \ref{theo-decay lower bound}, we have:

 \begin{prop}\label{prop-basic property}
Let $(M,g,f)$ be a $4$-dimensional
steady gradient Ricci soliton that dimension reduces to $3$-manifolds.
If  $(M,g,f)$ does not weakly dimension reduce to a steady gradient Ricci $3$-soliton, then there exist positive constants $r_0$, $C_1$, $C_2$, $C_3$, $c_1$ and $c_2$ such that
 \begin{align}
 C_1\rho(x)&\le f(x)\le C_2\rho(x)\,\text{ for all }\,\rho(x)\ge r_0,\\
 \frac{c_1}{\rho(x)}&\le R(x)\le \frac{c_2}{\rho(x)}\,\text{ for all }\,\rho(x)\ge r_0,\\
 |{\Rm}|(x)&\le C_3 R(x)\,\text{ for all }\,x\in M.
 \end{align}
 \end{prop}

 Proposition \ref{prop-basic property} implies that the level set of potential function $f(x)$ is compact and $|\nabla f|^2(x)>0$ when $\rho(x)\ge r_0$. Hence, we are able to use the level set flow (see Section 3 in \cite{DZ5}) to give an estimate of the diameter of the level set $\{x\in M:f(x)=r\}$ when $r$ is large enough. In \cite{DZ5}, the nonnegativity of the sectional curvatures is only used to get the following estimate
 \begin{align}\label{estimate-1}
\left| \Rm(\frac{\nabla f}{|\nabla f|},Y,Y,\frac{\nabla f}{|\nabla f|}) \right| \le \; & {\rm Ric}(\frac{\nabla f}{|\nabla f|},\frac{\nabla f}{|\nabla f|})
=\frac{|\langle \nabla R,\nabla f\rangle|}{|\nabla f|^2}\le C_0R^{\frac{3}{2}}\notag.
\end{align}
where $Y$ is a unit vector tangent to the level set. By Lemma 3.1 in \cite{DZ6}, we have
\begin{align}
{\rm Rm}(\nabla f,e_j,e_k,\nabla f)=-\frac{1}{2}({\rm Hess}R)_{jk}-R_{jl}R_{kl}+\Delta R_{jk}+2R_{ijkl}R_{il} .
\end{align}
Under the hypotheses of Theorem \ref{theorem-1}, it is easy to see by contradiction argument that
\begin{align}\label{estimate for gradient}
\frac{|\nabla^k{\Rm}|(x)}{R^{\frac{k+2}{2}}(x)}\le C(k)\,\text{ for all }\,\rho(x)\ge r_0.
\end{align}
Hence, we have
\begin{align}\label{estimate-2}
\left| {\Rm}(\frac{\nabla f}{|\nabla f|},Y,Y,\frac{\nabla f}{|\nabla f|})\right|\le\frac{|{\rm Hess }R|+|\Delta {\rm Ric}|+|{\rm Ric}|^2+2|{\rm Rm}|\cdot|{\rm Ric}|}{|\nabla f|^2}\le C_0R^2,
\end{align}
where $Y$ is a unit vector tangent to the level set.

Hence, under the hypotheses of Theorem \ref{theorem-1}, we can replace (\ref{estimate-1}) with (\ref{estimate-2}) to get the following diameter estimate for the level sets.

\begin{prop}\label{prop-diameter estimate}
Let $(M,g,f)$ be a $4$-dimensional
steady gradient Ricci soliton that dimension reduces to $3$-manifolds.
If  $(M,g,f)$ does not weakly dimension reduce to a
steady gradient Ricci $3$-soliton, then
there exists a constant $C_5$ independent of $r$ such that
\begin{align}
{\rm diam}(\Sigma_r,g)\le C_5\sqrt{r}\,\text{ for all }\,r\ge r_0,
\end{align}
where $\Sigma_r=\{x\in M:f(x)=r\}$.
\end{prop}

 With the help of Proposition \ref{prop-basic property}, Proposition \ref{prop-diameter estimate} and (\ref{estimate for gradient}), we can use the argument in Sections 2 through 4 of \cite{DZ5} to prove
a weak version of Theorem 1.3 in \cite{DZ5}.

\begin{theo}\label{theo-convergence}
Let $(M,g,f)$ be a $4$-dimensional
steady gradient Ricci soliton that dimension reduces to $3$-manifolds.
Suppose that $(M,g,f)$ does not weakly dimension reduce to a
steady gradient Ricci $3$-soliton. Then for any $p_{i}\rightarrow \infty$ the
 rescaled Ricci flows $(M,R(p_i)g(R^{-1}(p_i)t),p_{i})$ converge   subsequentially to
$(\mathbb{R}\times \Sigma,ds^2+g_{\Sigma}(t))$,  $t\in (-\infty,0]$, in the Cheeger--Gromov topology, where   $\Sigma$ is diffeomorphic to  a  level set $\Sigma_{r_0}$  and
 $(\Sigma,$ $g_{\Sigma}(t))$ is a $3$-dimensional compact ancient solution to the Ricci flow. Moreover, the scalar curvature $R_{\Sigma}(x,t)$ of $(\Sigma,g_{\Sigma}(t))$ satisfies
 \begin{align}\label{decay of t}
 R_{\Sigma}(x,t)\le\frac{C}{|t|}\,\text{ for all }\,x\in \Sigma,~t<0,
 \end{align}
 where $C$ is a
 constant.
\end{theo}

We note that $\Sigma$ in Theorem \ref{theo-convergence} is independent of the choice of the sequence $\{p_i\}$ and the subsequence since it is diffeomorphic to the level set $\Sigma_{r_0}$ of the potential $f$ for some $r_0$. We also note that $\Sigma$ is connected. This is due to the following theorem in \cite{MW}.
\begin{theo}[Munteanu and Wang] 
A complete noncompact steady gradient Ricci soliton is either connected at infinity (i.e., has exactly one end) or splits as the product of $\mathbb{R}$ with a compact Ricci flat manifold. Hence, a complete noncompact non-Ricci-flat  steady gradient Ricci soliton must be connected at infinity.
\end{theo}

By Theorem \ref{theo-linear of f}, we know that $M=S(\varepsilon_0)\cup M\setminus S(\varepsilon_0)$ and $M\setminus S(\varepsilon_0)$ is diffeomorphic to $\Sigma\times (A,+\infty)$.  Since $(M,g)$ has only one end, $\Sigma$ must be connected.

To prove Theorem \ref{theorem-1}, we are left to show that $\Sigma$ in Theorem \ref{theo-convergence} is diffeomorphic to $\mathbb{S}^3/\Gamma$ and that
$g_{\Sigma}(t)$ is a
family
of round metrics. Moreover, we need to show that the limit steady gradient Ricci soliton must be the product of a line and the Bryant $3$-soliton if it weakly dimension reduces to a steady gradient Ricci $3$-soliton. These results will be proved in the next two sections (see Theorem \ref{theo-uniqueness of limit ancient flow} and Theorem \ref{theo-uniqueness of limit soliton}).

\section{Uniqueness of the limit ancient Ricci flow}\label{section-Uniqueness of limit}

\begin{theo}\label{theo-uniqueness of limit ancient flow}
The $3$-manifold
$\Sigma$ in Theorem \ref{theo-convergence} is diffeomorphic to $\mathbb{S}^3/\Gamma$ and $g_{\Sigma}(t)$ is a
family
of shrinking round metrics.
\end{theo}

We first note that the following lemma holds
because of the dimension reduction assumption.

\begin{lem}\label{lem-volume of geodesic ball}
Let $(M,g,f)$ be a $4$-dimensional
steady gradient Ricci soliton that dimension reduces to $3$-manifolds.
Then there exist positive constants $\kappa$ and $r_0$ such that (noncollapsing at the curvature scale of a point)
\begin{align}\label{volume of geodesic ball}
{\rm vol}\,B(p,1;R(p)g)\ge \kappa\,\text{ for all }p\in M\text{ such that }\rho(p)\ge r_0.
\end{align}
For any positive
number
$\bar{r}$, there exists a positive constant $C(\bar{r})$ such that (bounded curvature at bounded distance)
\begin{align}\label{curvature estimate in geodesic ball}
\frac{R(x)}{R(p)}\le C(\bar{r})\,\text{ for all }\,x\in B(p,\bar{r};R(p)g)
\text{ and }p\text{ such that }
\rho(p)\ge r_0.
\end{align}
\end{lem}
\begin{proof}
The proof is
by contradiction. Suppose that the lemma is not true. Then, there exists a sequence of points $p_i$ tending to infinity such that
\begin{align}\label{volume to zero}
{\rm vol}\,B(p_i,1;R(p_i)g)\to0 ~{\rm as}~p_i\to\infty.
\end{align}
On the other hand, we may assume that $(M,R(p_i)g(R^{-1}(p_i)t),p_i)$
subconverges
to $(M_{\infty},g_{\infty}(t),p_{\infty})$. By taking $t=0$, $(M,R(p_i)g,p_i)$ converges to $(M_{\infty},g_{\infty}(0),p_{\infty})$. Therefore, for $i$ large, we have
\begin{align}
{\rm vol}\,B(p_i,1;R(p_i)g)\ge \frac{1}{2}{\rm vol}\,B(p_{\infty},1;g_{\infty}(0))>0.
\end{align}
This contradicts (\ref{volume to zero}).

Similarly, one can prove (\ref{curvature estimate in geodesic ball}) by a contradiction argument.
This completes the proof.

\end{proof}

In \cite{DZ5}, X.-H.~Zhu and the second author
used the noncollapsing condition to obtain (\ref{volume of geodesic ball}). By Lemma \ref{lem-volume of geodesic ball}, we obtain (\ref{volume of geodesic ball}) without assuming the noncollapsing condition. Therefore, we can follow the argument of Lemma 4.2 and Lemma 4.3 in \cite{DZ5} to obtain the following volume  estimate for level sets.

\begin{lem}\label{lem-lower volume bound of level set}
Let $(M,g,f)$ be a $4$-dimensional
steady gradient Ricci soliton that dimension reduces to $3$-manifolds, but does not weakly dimension reduce to a
steady gradient Ricci $3$-soliton.
Then
there exists a positive constant $\kappa_1$ independent of $r$ such that
\begin{align}
{\rm vol}(\Sigma_r,g)\ge \kappa_1(\sqrt{r})^{3}\,\text{ for all }\,r\ge r_0,
\end{align}
where $\Sigma_r=\{x\in M:f(x)=r\}$.

\end{lem}

As a corollary, we have the following lemma.

\begin{lem}
Let $(M,g,f)$ be a $4$-dimensional
steady gradient Ricci soliton that dimension reduces to $3$-manifolds, but does not weakly dimension reduce to a
steady gradient Ricci $3$-soliton.
Then there exist positive constants $\kappa_2$ and $C_6$ such that
\begin{align}\label{volume lower bound}
{\rm vol }(\Sigma,g_{\Sigma}(t))\ge \kappa_2(-t)^{\frac{3}{2}}.
\end{align}
and
\begin{align}\label{diameter estimate of sigma}
{\rm diam }(\Sigma,g_{\Sigma}(t))\le C_6\sqrt{1-t}.
\end{align}
\end{lem}

\begin{proof}
Fix $t<0$. Under the hypotheses of Theorem \ref{theo-convergence}, for $p_i$ tending to infinity, we assume the following convergence
\begin{align}
(M,R(p_i)g(R^{-1}(p_i)t),p_i)\to (\mathbb{R}\times \Sigma,ds^2+g_{\Sigma}(t),p_{\infty}).
\end{align}
Therefore, we have
\begin{align}
(M,R(p_i)g,\phi_{(R^{-1}(p_i)t)}(p_i))\to (\mathbb{R}\times \Sigma,ds^2+g_{\Sigma}(t),p_{\infty}).
\end{align}

Let $q_i=\phi_{R^{-1}(p_i)t}(p_i)$. Similar to Corollary 3.4 in \cite{DZ5}, there exists a positive constant $C_0^{\prime}$ such that
\begin{align}
\Sigma_{f(q_i)}\subseteq B(q_i,C_0^{\prime};R(q_i)g).
\end{align}
By (\ref{curvature estimate in geodesic ball}), we get
\begin{align}
\frac{R(x)}{R(q_i)}\le C(C_0^{\prime})\,\text{ for all }\,x\in \Sigma_{f(q_i)}.
\end{align}
Note that
\begin{align}
R(p_i)f(q_i)=&R(p_i)\left(f(p_i)+\int_{0}^{R^{-1}(p_i)|t|}|\nabla f|^2(\phi_s(p_i)){\rm d}s\right)\ge c_1C_1+\varepsilon_0|t|.
\end{align}
Hence,
\begin{align}\label{curvature estimate in level set}
\frac{R(x)}{R(p_i)}=&\frac{R(x)}{R(q_i)}\frac{R(q_i)f(q_i)}{R(p_i)f(q_i)}\notag\\
\le &\frac{C_1^{\prime}}{c_1C_1+\varepsilon_0|t|}\notag\\
\le&\frac{C_2^{\prime}}{1+|t|}\,\text{ for all }\,x\in \Sigma_{f(q_i)}.
\end{align}
By Lemma \ref{lem-lower volume bound of level set}, we have the following volume estimate:
\begin{align}\label{volume bound of level set}
{\rm vol}\left({ \Sigma_{f(q_i)}},R(p_i)g\right)=&{\rm vol}\left({ \Sigma_{f(q_i)}},f^{-1}(q_i)g\right)\cdot\big(R(q_i)f(q_i)\big)^{\frac{3}{2}}\cdot\Big(\frac{R(p_i)}{R(q_i)}\Big)^{\frac{3}{2}}\notag\\
\ge&\kappa_1\cdot(c_1C_1)^{\frac{3}{2}}\cdot\Big(\frac{1+|t|}{C_2^{\prime}}\Big)^{\frac{3}{2}}\notag\\
\ge &\kappa_2(-t)^{\frac{3}{2}}.
\end{align}

Moreover, by Proposition \ref{prop-diameter estimate}, we have the diameter estimate for the level set:
\begin{align}
{\rm diam}(\Sigma_{f(q_i)},g)\le C_5\sqrt{f(q_i)}.\notag
\end{align}
Therefore,
\begin{align}\label{diameter of level set}
{\rm diam}(\Sigma_{f(q_i)},R(p_i)g)\le& \, C_5\sqrt{R(p_i)f(q_i)}\notag\\
= & \, C_5\sqrt{R(p_i)\left(f(p_i)+\int_{0}^{R^{-1}(p_i)|t|}|\nabla f|^2(\phi_s(p_i)){\rm d}s\right)}\notag\\
\le & \, C_5\sqrt{c_2C_2+R_{\max}(-t)}.
\end{align}

Similar to \cite{DZ5}, we can use (\ref{curvature estimate in level set}), (\ref{volume bound of level set}) and (\ref{diameter of level set}) to show that
\begin{align}\label{convergence of level set}
(\Sigma_{f(q_i)},R(p_i)g)\to (\Sigma,g_{\Sigma}(t)).
\end{align}
 Hence, we get the following estimates by (\ref{volume bound of level set}), (\ref{diameter of level set}) and the convergence (\ref{convergence of level set}):
\begin{align}
{\rm vol}(\Sigma,g_{\Sigma}(t))=\lim_{i\to \infty}{\rm vol}({ \Sigma_{f(q_i)}},R(p_i)g)\ge \kappa_2(-t)^{\frac{3}{2}} \notag
\end{align}
and
\begin{align}
{\rm diam}(\Sigma,g_{\Sigma}(t))=\lim_{i\to \infty}{\rm diam}({ \Sigma_{f(q_i)}},R(p_i)g)\le C_5\sqrt{c_2C_2+R_{\max}(-t)}.\notag
\end{align}

\end{proof}

Now, we are ready to prove Theorem \ref{theo-uniqueness of limit ancient flow}.
\begin{proof}[Proof of Theorem  \ref{theo-uniqueness of limit ancient flow}]
For fixed $q\in M$, by the curvature estimate (\ref{decay of t}) and the volume estimate (\ref{volume lower bound}), we can use the argument in \cite{Na}
 to show that there exists $\tau_i\to+\infty$ such that
 \begin{align}
 (\Sigma,\tau_i^{-1}g(\tau_i t),q)\to(\Sigma_{\infty},h(t),q_{\infty}),~t<0.
 \end{align}
 Moreover, $(\Sigma_{\infty},h(t),q_{\infty})$ is a gradient shrinking Ricci soliton. Hence, $\Sigma_{\infty}$ must be a finite quotient of $\mathbb{R}^3$, $\mathbb{S}^2\times\mathbb{R}$ or $\mathbb{S}^3$.  By the diameter estimate (\ref{diameter estimate of sigma}), we know that $\Sigma_{\infty}$ is diffeomorphic to $\Sigma$. Hence, $\Sigma_{\infty}$ is compact. Therefore, $\Sigma_{\infty}$ must be a finite quotient of $\mathbb{S}^3$ and $h(t)$ is a round metric. Hence, $\Sigma=\mathbb{S}^3/\Gamma$. Let $\nu(\Sigma,g_{\Sigma}(t))$ be the entropy introduced by Perelman \cite{Pe1}. Note that it is non-decreasing in $t$. Hence, we know $\nu(\Sigma,g(t))\ge \nu(\Sigma_{\infty},h(t))$. By \cite{Ch} and the topology of $\Sigma$, $(\Sigma,g_{\Sigma}(t))$ has positive sectional curvature.  By \cite{H1}, $(\Sigma,g_{\Sigma}(t))$ blows up at time $T$ and the flow converges to a round metric under rescaling when $t\to T$. Therefore,  $\nu(\Sigma,g_{\Sigma}(t))= \nu(\Sigma_{\infty},h(t))$. Hence, $(\Sigma,g_{\Sigma}(t))$ is a three-dimensional compact gradient shrinking Ricci soliton. We have completed the proof.
\end{proof}

\section{The limit soliton is the Bryant 3-soliton}\label{section-limit soliton is Bryant}

\begin{theo}\label{theo-uniqueness of limit soliton}
Let $(M,g,f)$ be a $4$-dimensional
steady gradient Ricci soliton that dimension reduces to $3$-manifolds.
Suppose that
$(M,g,f)$ weakly dimension reduces to a
steady gradient Ricci $3$-soliton $(N_{\infty}^3,h_{\infty})$.
Then $(N_{\infty},h_{\infty})$ must be isometric to the Bryant $3$-soliton.
\end{theo}

\begin{lem}\label{lem-reduction}
Let $(M,g,f)$ be a $4$-dimensional steady gradient Ricci soliton which dimension reduces to $3$-manifolds.
Suppose that $(M,g,f)$ weakly dimension reduces to a
steady gradient Ricci $3$-soliton $(N^3,g_N,f_N)$.
Then $(N,g_N,f_N)$ dimension reduces to $2$-manifolds.
\end{lem}
\begin{proof}
We may assume that $(M,g,f)$ dimension reduces to $(N,g_N,f_N)$ along points $p_i$ tending to infinity. Precisely, we have that
\begin{align}\label{convergence-1}
(M,R(p_i)g(R^{-1}(p_i)t),p_i)\to (M_{\infty},g_{\infty}(t),p_{\infty}),
\end{align}
where $(M_{\infty},g_{\infty}(t))=(N\times\mathbb{R},g_{N}(t)+{\rm d}s^2)$.

\begin{claim}\label{claim-convergence}
For any $q_k\in N$ tending to infinity, by taking a subsequence, we have the convergence
\begin{align}
(N,R_N(q_k,0)g(R^{-1}_N(q_k,0)t),q_k)\to (N_{\infty},\bar{g}_{\infty}(t),q_{\infty}),
\end{align}
where $(N_{\infty},\bar{g}_{\infty}(t))$ is a $3$-dimensional complete ancient flow.
\end{claim}
Let $\hat{q}_k=(q_k,0)\in N\times \mathbb{R}$ for $k\in \mathbb{N}$.  By convergence (\ref{convergence-1}), for any fixed $k\in \mathbb{N}$, there exists a sequence of points $q_{k,i}\in M$ such that
\begin{align}\label{convergence-2}
(M,R(p_i)g(R^{-1}(p_i)t),q_{k,i})\to (M_{\infty},g_{\infty}(t),\hat{q}_k).
\end{align}
The convergence above implies that
\begin{align}\label{curvature convergence}
\frac{R(q_{k,i})}{R(p_i)}\to R_{\infty}(\hat{q}_k,0),~{\rm as}~i\to\infty.
\end{align}
By (\ref{convergence-2}) and (\ref{curvature convergence}), we have
\begin{align}\label{convergence-3}
(M,R(q_{k,i})g(R^{-1}(q_{k,i})t),q_{k,i})\to(M_{\infty},R_{\infty}(\hat{q}_k,0)g(R_{\infty}^{-1}(\hat{q}_k,0)t),\hat{q}_k)).
\end{align}
By Lemma \ref{lem-volume of geodesic ball}, for fixed $\bar{r}>0$, we have
\begin{align}\label{curvature bound}
\frac{R(x)}{R(q_{k,i})}\le C(2\bar{r})\,\text{ for all }\,x\in B(q_{k,i},2\bar{r};R(q_{k,i})g),~\rho(q_{k,i})\ge r_0,
\end{align}
and
\begin{align}\label{volume bound}
{\rm vol}\, B(q_{k,i},1;R(q_{k,i})g)\ge \kappa.
\end{align}
By (\ref{convergence-3}), (\ref{curvature bound}) and (\ref{volume bound}), we have
\begin{align}
R_{\infty}(x)\le  C(2\bar{r})\,\text{ for all }\,x\in B(\hat{q}_k,\bar{r};R_{\infty}(\hat{q}_k,0)g_{\infty}(0)).
\end{align}
and
\begin{align}
{\rm vol} \, B(\hat{q}_k,1;R_{\infty}(\hat{q}_k,0)g_{\infty}(0))\ge \kappa.
\end{align}
It follows that
\begin{align}\label{upper bound}
R_{N}(x,0)\le C(2\bar{r})\,\text{ for all }\,x\in B(q_k,\bar{r};R_{N}(q_k,0)g_{N}(0)),
\end{align}
and
\begin{align}\label{volume}
{\rm vol} \, B(\hat{q}_k,1;R_{\infty}(\hat{q}_k,0)g_{\infty}(0))\ge \frac{\kappa}{2}.
\end{align}
Note that $(N,g_N,f_{N})$ is a $3$-dimensional steady gradient Ricci soliton. Then, $(N,g_N)$ has nonnegative sectional curvature by \cite{Ch}. Note that $g_{N}(t)$ is a Ricci flow generated by $(N,g_N,f_{N})$. Therefore,
\begin{align}\label{monotone of R}
\frac{\partial R_{N}(x,t)}{\partial t}=2{\Ric}_{N}(\nabla f_{N}, \nabla f_{N})(\varphi_t(x))\ge0,
\end{align}
where $\varphi_t$ is generated by $-\nabla f_{N}$.

 By (\ref{upper bound}) and (\ref{monotone of R}), we get
\begin{align}\label{curvature estimate}
|{\Rm}_{N}|(x,t)\le C(n)R_{N}(x,t)\le C(n)C(2\bar{r})\quad \forall \,x\in B(q_k,\bar{r};R_{N}(q_k,0)g_{N}(0)).
\end{align}
Let $h_k(t)=R_{N}(q_k,0)g_{N}(R_{N}^{-1}(q_k,0)t)$. Then, $h_{k}(t)$ satisfies the Ricci flow equation. Note that the sectional curvatue of $h_{k}(t)$ is nonnegative. Hence, 
\begin{align*}
    h_{k}(x,t)\ge h_{k}(x,0)\quad\forall~t\le0,~x\in N.
\end{align*}
Therefore,
\begin{align*}
    B(q_k,\bar{r};h_k(t))\subset  B(q_k,\bar{r};h_k(0)), ~\forall~t\le0.
\end{align*}
Finally, for $t\le 0$, (\ref{curvature estimate}) implies
\begin{align}\label{curvature estimate-new}
|{\Rm}_{N}|(x,t)\le C(n)C(2\bar{r})\,\text{ for all }\,x\in B(q_k,\bar{r};h_k(t)).
\end{align}
For any $\bar{r}>0$, we can find constant $C(2\bar{r})$ such that (\ref{curvature estimate-new}) holds. By Theorem 1.7 in \cite{To14}, Claim \ref{claim-convergence} follows from (\ref{volume}) and (\ref{curvature estimate-new}). 

Let $(N_{\infty},\bar{g}_{\infty}(t))$ be the ancient Ricci flow in Claim \ref{claim-convergence}. We are left to show the following claim.
\begin{claim}\label{claim-split}
$(N_{\infty},\bar{g}_{\infty}(t))=(S\times\mathbb{R},g_{S}(t)+{\rm d}s^2)$, where $(S,g_{S}(t))$ is a two-dimensional ancient flow with bounded curvature.
\end{claim}
For any sequence $q_k$ tending to infinity, we may assume that $R_{N}(q_k,0)\to A$ by taking a subsequence. If $A>0$, then $R_{N}(q_k,0)d^2_{g_{N}(0)}(q_k,q_0)\to \infty$ as $k\to\infty$. By Theorem 5.35 in \cite{MT}, we get $(N_{\infty},\bar{g}_{\infty}(t))=(S\times\mathbb{R},g_{S}(t)+{\rm d}s^2)$, where $(S,g_{S}(t))$ is a two-dimensional ancient flow. By (\ref{monotone of R}), $(S,g_{S}(t))$ also satisfies
\begin{align}
\frac{\partial R_{S}(x,t)}{\partial t}\ge 0\,\text{ for all }\,x\in S.
\end{align}
Now, we need to show that $R_{S}(x,t)$ has bounded curvature for all $x\in S$ and $t\le0$. It is sufficient to show $R_S(x,0)$ is bounded for $x\in S$. The idea is similar to the proof of Lemma 4.4 in \cite{DZ2}. If it is not true, then there exists $x_i\to\infty$ such that $R_S(x_i,0)\to \infty$. By Claim \ref{claim-convergence} and the argument in the proof of Claim \ref{claim-convergence}, we have the following convergence by taking a subsequence
\begin{align}
(S,R_S(x_i,0)g(R^{-1}_S(x_i,0)t),x_i)\to (S_{\infty},g^{\prime}_{\infty}(t),x_{\infty}).
\end{align}
Note that $R^{\prime}_{\infty}(x_{\infty},0)=1$.

On the other hand, we note that $R_S(x_i,0)d^2_{g_S(0)}(x_i,x_0)\to\infty$ and therefore $(S_{\infty},g^{\prime}_{\infty}(t))$ splits off a line by Theorem 5.35 of Chapter 5 in \cite{MT}. Since $S_{\infty}$ is a two-dimensional manifold, $(S_{\infty},g^{\prime}_{\infty}(t))$ must be flat. Then, $R^{\prime}_{\infty}(x_{\infty},0)=0$. This is impossible. Hence, we have shown that $R_{S}(x,t)$ has bounded curvature.

We are left to deal with the case $A=0$. We follow the notation in Claim \ref{claim-convergence} and assume the following convergence for  $q_k\in N$ tending to infinity
\begin{align}
(N,R_N(q_k,0)g(R^{-1}_N(q_k,0)t),q_k)\to (N_{\infty},\bar{g}_{\infty}(t),q_{\infty}),
\end{align}

 Let $g_k(t)=R_N(q_k,0)g(R^{-1}_N(q_k,0)t)$ and $X_{(k)}=R_N(q_k,0)^{-\frac{1}{2}}\nabla f_{N}$. For any fixed $\bar{r}>0$, we have
\begin{align}
\sup_{ B(q_{k},\bar r ;  {g_{k}(0)})}| \nabla X_{(k)}|_{g_{k}(0)}&= \sup_{ B(q_{k},\bar r ;  {g_{k}(0)})}\frac{|{\rm Ric}_{N}|_{g_N}}{\sqrt{R_N(q_{k})}}\le C\sqrt{R_{N}(q_{k})} \to 0.\notag
\end{align}
Similarly,
$$\sup_{ B(q_{k},\bar r ;  {g_{k}(0)})}| \nabla^{m}X_{(k)}|_{g_{k}(0)}\leq C(n)\sup_{ B(q_{k},\bar r ;  {g_{k}(0)})}| \nabla^{m-1}{\rm Ric}_{g_{k}(0)}|_{g_{k}(0)}\le C_1.$$
Thus  $X_{(k)}$ converges  subsequentially  to a parallel  vector field $X_{(\infty)}$ on $( M_{\infty},$ $  g_{\infty}(0))$.
 Moreover,
 \begin{align}
 |X_{(i)}|_{g_{k}(0)}( x)=|\nabla f_{N}|_{g_N}(x)
=\sqrt{R_{\rm max}}+o(1)>0\,\text{ for all }\,x\in B(p_{i},\bar r ;  {g_{i}}),\notag
 \end{align}
as long as $f(p_i)$ is large enough.  This implies that $X_{(\infty)}$ is non-trivial.
 Hence,  $( N_{\infty},\bar{g}_{\infty}(t))$ locally splits off a line along $X_{(\infty)}$. It is not hard to show that the integral curve of $X_{(\infty)}$ is a line. So $( M_{\infty},\bar{g}_{\infty}(t))$ splits off a line globally.

Now, we already have $(N_{\infty},\bar{g}_{\infty}(t))=(S\times\mathbb{R},g_{S}(t)+{\rm d}s^2)$, where $(S,g_{S}(t))$ is a two-dimensional ancient flow. We can use the argument in the case $A>0$ to show that $(S,g_{S}(t))$ has uniformly bounded curvature.

\end{proof}

\begin{lem}\label{lem-uniform decay}
$(N,g_{N},f_{N})$ has a uniform curvature decay.
\end{lem}

\begin{proof}
We prove this by contradiction. If the lemma is not true, then there exists a sequence $q_i$ tending to infinity such that $R_{N}(q_i)\ge C_0$ for some positive constant $C_0$ and all $i\in\mathbb{N}$. Since $R_N(x)$ is bounded, we may assume that  $R_{N}(q_i)\to C_0$ as $i\to\infty$. By a delicate choice of $q_i$, we can even make sure that $C_0=\lim_{r\to\infty}\sup_{x\in N\setminus B(q_0,r;g_N)}R_{N}(x)$. By Lemma \ref{lem-reduction}, we may assume the following convergence
\begin{align}
(N,R_N(q_i,0)g(R^{-1}_N(q_i,0)t),q_i)\to (N_{\infty},\bar{g}_{\infty}(t),q_{\infty}),
\end{align}
where $(N_{\infty},\bar{g}_{\infty}(t))=(S\times\mathbb{R},g_{S}(t)+{\rm d}s^2)$ and $(S,g_{S}(t))$ is a two-dimensional ancient flow with bounded curvature.

Hence, $R_{N_{\infty}}(q_{\infty},0)=1$ and $R_{N_{\infty}}(x,t)\le1$ for all $x\in N_{\infty}$ and $t\in (-\infty,+\infty)$. Note that $R_{N_{\infty}}(\widehat{x},t)=R_{S}(x,t)$ for $\widehat{x}=(x,0)\in S\times\mathbb{R}$. Therefore, $(S,g_{S}(t))$ must be a Ricci flow generated by the cigar soliton. Since the cigar soliton is collapsed, for any $x_i\in S$ tending to infinity, we have
\begin{align}
\lim_{i\to\infty}{\rm vol} \, B(x_i,1;R_{S}(x_i,0)g_{S}(0))=0
\end{align}
Let $\widehat{x}_i=(x_i,0)\in S\times\mathbb{R}$. It follows that
\begin{align}
\lim_{i\to\infty}{\rm vol} \, B(\widehat{x}_i,1;R_{N_{\infty}}(\widehat{x}_i,0)g_{N_{\infty}}(0))=0
\end{align}
On the other hand, by Claim \ref{claim-convergence} as wells as the argument in the proof of Claim \ref{claim-convergence}, by taking a subsequence, we have
\begin{align}
{\rm vol} \, B(\widehat{x}_i,1;R_{N_{\infty}}(\widehat{x}_i,0)g_{N_{\infty}}(0))\ge c_0,
\end{align}
for some positive constant $c_0$. This is impossible. Hence, we have completed the proof.

\end{proof}

\begin{lem}\label{lem-estimate for linear decay}
 There exists a constant $\epsilon$ such that
\begin{align}\label{estimate for decay}
\epsilon\le \frac{\Delta_{N} R_{N}(x)+2|{\Ric}_{N}|_{N}^2(x)}{R_{N}^2(x)}\le C_4\,\text{ for all }\,x\in N\setminus K,
\end{align}
where $K$ is a compact set.
\end{lem}
\begin{proof}
Claim \ref{claim-convergence} implies the inequality on the right-hand side of the lemma. Similar to the proof of Lemma \ref{lem-estimates from asymptotical geometry}, $(N,g_N,f_N)$ dimension reduces to a two-dimensional steady gradient Ricci soliton if the left-hand side of (\ref{estimate for decay}) is not true. However, by the argument in Lemma \ref{lem-uniform decay}, $(N,g_N,f_N)$ cannot dimension reduce to a two-dimensional steady gradient Ricci soliton. Hence, the inequality on the left-hand side of (\ref{estimate for decay}) holds.
\end{proof}

Now, we are ready to prove Theorem \ref{theo-uniqueness of limit soliton}.
\begin{proof}[Proof of Theorem \ref{theo-uniqueness of limit soliton}]
By Lemma \ref{lem-uniform decay}, Lemma \ref{lem-estimate for linear decay}, Theorem \ref{theo-decay} and Theorem \ref{theo-decay lower bound}, we have
\begin{align}
\frac{c^{\prime}}{\rho_N(x)}\le R_{N}(x)\le \frac{c}{\rho_{N}(x)}\,\text{ for all }\,\rho(x)\ge r_0,
\end{align}
for some positive constants $c^{\prime}$, $c$ and $r_0$. By the result in \cite{DZ3}, $(N,g_{N},f_{N})$ must be isometric to the Bryant $3$-soliton.
\end{proof}

\section{Volume growth of steady GRS}\label{section-volume growth}

In this section, we prove Theorem \ref{theo-maximal volume growth}.

\begin{lem}\label{set-mr-contain-1}
Let $(M,g,f)$ be a steady gradient Ricci soliton.  For  any $p\in M$ and number $k>0$, we have
\begin{align}\label{set-mr-contain}
B\left(p,\tfrac{k}{\sqrt{R_{\max}}}; M(p) g\right)\subset M_{p,k},
\end{align}
where $M_{p,k}=\left\{x\in M:f(p)-\frac{k}{\sqrt{M(p)}}\le f(x)\le f(p)+\frac{k}{\sqrt{M(p)}}\right\}$.
\end{lem}

\begin{proof}
For any $q\in M$, let $\gamma(s)$ be any curve connecting $p$ and $q$ such that $\gamma(s_{1})=q$ and $\gamma(s_{2})=p$. Then,
\begin{align*}
\mathcal{L}(q,p)=&\int_{s_{1}}^{s_{2}}\sqrt{\langle \gamma^{\prime}(s),\gamma^{\prime}(s)\rangle}ds\\
\geq& \int_{s_{1}}^{s_{2}}\frac{|\langle\gamma^{\prime}(s),\nabla f\rangle|}{|\nabla f|}ds\\
\geq& \frac{1}{\sqrt{R_{\max}}} \left|\int_{s_{1}}^{s_{2}}\langle\gamma^{\prime}(s),\nabla f\rangle ds\right| \\
=&\frac{1}{\sqrt{R_{\max}}}|f(p)-f(q)|,
\end{align*}
where we have used
$$|\nabla f|^2(x)=R_{\max}-R(x)\le R_{\max}\quad\forall~x\in M.$$
It follows that
\begin{align}
d(q,p)\geq \frac{1}{\sqrt{R_{\max}}}|f(p)-f(q)|.\notag
\end{align}
In particular,  for  $q\in M\setminus M_{p,k}$, we get
\begin{align}
d(q,p)\geq \frac{1}{\sqrt{R_{\max}}}\cdot\frac{k}{\sqrt{M(p)}}.\notag
\end{align}
Hence
\begin{align}
B\left(p,\tfrac{k}{\sqrt{R_{\max}}};M(p)g\right)\subset M_{p,k}.
\end{align}

\end{proof}

\begin{lem}\label{lem-monotone of S}
Let $(M,g,f)$ be a steady gradient Ricci soliton with uniform scalar curvature decay which satisfies condition \eqref{Ricci nonnegative outside K}. If $f(x)\ge r$, then
\begin{align}
R(x)\le \sup_{y\in \Sigma_r}R(y)\quad\forall~r\ge r_0,
\end{align}
where $r_0$ is a positive constant.
\end{lem}
\begin{proof}
By Theorem \ref{theo-linear of f}, we may assume that there exist constants $C_1$ and $C_2$ such that
\begin{align}
C_1\rho(x)\le f(x)\le C_2 \rho(x), ~\forall~f(x)\ge r_0.
\end{align}
Let $S(\varepsilon_0)$ be the compact set in Lemma \ref{lem-remain for t negative}. We may also assume that $S(\varepsilon_0)\subseteq\{x\in M: f(x)\ge r_0\}$. Similar to the proof of Lemma \ref{lem-shrinking}, we can find $t_x\ge 0$ such that $f(\phi_{t_x}(x))=r$, where $\phi_t$ is generated by $-\nabla f$. Since the Ricci curvature is nonnegative, we have
\begin{align}
R(x)\le R(\phi_{t_x}(x))\le \sup_{y\in \Sigma_r}R(y).
\end{align}
We have completed the proof.
\end{proof}

Now, we let $M(p)=\sup_{x\in\Sigma_{\frac{f(p)}{2}}}R(x)$. By Lemma \ref{set-mr-contain-1}, we can prove:

\begin{lem}\label{lem-pointwise curvature estimate}
Let $(M,g,f)$ be a steady gradient Ricci soliton with uniform scalar curvature decay which satisfies condition \eqref{Ricci nonnegative outside K}.  Fix $\epsilon>0$. Then, for   any $p\in M$ with $M(p)\ge \frac{\epsilon}{f(p)}$ and number $k>0$, there exist constants $r_0(k,\epsilon)$ and $C(m)$ such that
\begin{align}\label{pointwise curvature estimate}
\frac{|\nabla^m {\Rm}|(q)}{M^{\frac{m+2}{2}}(p)}\le C(m)\quad\forall~q\in ~M_{p,k}, ~f(p)\ge r_0(k,\epsilon).
\end{align}
\end{lem}

\begin{proof}
Fix any $q\in M_{p,k}$ with $f(p)\ge r_0 \gg 1$. Similar to the proof of Lemma \ref{set-mr-contain-1}, we have
\begin{align}
B\left(q,\tfrac{1}{\sqrt{R_{\max}}};M(p)g\right)\subseteq& \left\{x\in M:f(q)-\tfrac{1}{\sqrt{M(p)}}\le f(x)\le f(q)+\tfrac{1}{\sqrt{M(p)}}\right\}\notag\\
\subseteq& \, M_{p,k+1}.\notag
\end{align}
Since $M(p)\ge \frac{\epsilon}{f(p)}$, for $r_0(k,\epsilon)$ large enough, we have
\begin{align}
f(x)\ge \frac{f(p)}{2}\quad\forall~x\in M_{p,k+1}.\notag
\end{align}
Hence, by Lemma \ref{lem-monotone of S}, we get
\begin{align}\label{curvature bound-1}
R(x)\le M(p)\quad\forall ~x\in B\left(q,\tfrac{1}{\sqrt{R_{\max}}};M(p)g\right).
\end{align}

Let $\phi_t$ be generated by $-\nabla f$. Then $g(t)=\phi_t^{\ast}g$ satisfies the Ricci flow,
\begin{align}\label{Ricci flow equation}
         \frac{\partial g(t)}{\partial t} &= -2{\rm Ric}(g(t)).
                  \end{align}
Also, the rescaled flow  $g_p(t)=M(p) g(M^{-1}(p) t)$  satisfies (\ref{Ricci flow equation}). Since the Ricci curvature is nonnegative,
\begin{align}
B\left(q,\tfrac{1}{\sqrt{R_{\max}}};g_p(t)\right)\subseteq B\left(q,\tfrac{1}{\sqrt{R_{\max}}};g_p(0)\right),~t\in~[-1,0].\notag
\end{align}
Combining this with  (\ref{curvature bound-1}) and the estimate in \cite{Chan}, we get
\begin{align}
|{\Rm}|_{g_p(t)}(x)\le CR_{g_p(t)}(x)\le C\quad\forall~x\in B\left(q,\tfrac{1}{\sqrt{R_{\max}}};g_p(0)\right),~t\in [-1,0].
\end{align}
Thus,  by Shi's higher order local derivative of curvature estimates, we obtain
\begin{align}
|\nabla^m_{g_p(t)} {\Rm}_{g_p(t)}|(x)\le C(m)\quad\forall~x\in B\left(q,\tfrac{1}{ 2\sqrt{R_{\max}}};g_p(-1)\right),~t\in  [-\frac{1}{2},0].\notag
\end{align}
It follows that
\begin{align}
|\nabla^m {\Rm}|(x)\le C(m)M^{\frac{m+2}{2}}(p)\quad\forall~x\in B\left(q,\tfrac{1}{2\sqrt{R_{\max}}};g_p(-1)\right).\notag
\end{align}
In particular, we have
\begin{align}
|\nabla^m {\Rm}|(q)\le C(m)M^{\frac{m+2}{2}}(p).\notag
\end{align}
The lemma is proved.
\end{proof}

\begin{lem}\label{lem-curvature fast decay}
Let $(M,g,f)$ be a $4$-dimensional steady gradient Ricci soliton with uniform scalar curvature decay which satisfies condition \eqref{Ricci nonnegative outside K}. Suppose $(M,g)$ has maximal volume growth.  Suppose $\{p_i\}_{i\in \mathbb{N}_{+}}\in M$ is a sequence of points with the property $M(p_i)\ge \frac{\epsilon}{f(p_i)}$, where $\epsilon$ is a given constant. If $p_i$ tends to infinity, then for any $\varepsilon>0$, there exists a constant $r_0(\varepsilon,\epsilon)$ such that
\begin{align}
R(p_i)\le \varepsilon M(p_i)\quad\forall~f(p_i)\ge r_0(\varepsilon,\epsilon).
\end{align}
\end{lem}

\begin{proof}
We prove this by contradiction. We only need to consider the case that $(M,g,f)$ is non-Ricci-flat. If the lemma is not true, then by taking a subsequence, we may assume that there exists a constant $\varepsilon>0$ such that
\begin{align}\label{positive-assumption}
R(p_i)>\varepsilon M(p_i)~\mbox{as}~i\to\infty.
\end{align}
Now we consider the sequence of Ricci flows $(M,M(p_i)g(M^{-1}(p_i)t),p_i)$. Let $g_i(t)=M(p_i)g(M^{-1}(p_i)t)$. Since we assume that $M(p_i)\ge \frac{\epsilon}{f(p_i)}$, by Lemma \ref{lem-pointwise curvature estimate}, we have
\begin{align}\label{8-1}
\frac{|\nabla^m {\Rm}|(x)}{M^{\frac{m+2}{2}}(p_i)}\le C(m)\quad\forall~x\in M_{p_i,k},~f(p_i)\ge r_0(k,\epsilon).
\end{align}
By Lemma \ref{set-mr-contain-1}, we also have
\begin{align}\label{8-2}
B\left(p_i,\tfrac{k}{\sqrt{R_{\max}}};g_i(0)\right)\subseteq M_{p_i,k} .
\end{align}
By (\ref{8-1}) and (\ref{8-2}), we have
\begin{align}\label{8-3}
\frac{|\nabla^m {\Rm}|(x)}{M^{\frac{m+2}{2}}(p_i)}\le C(m)\quad\forall~x\in B\left(p_i,\tfrac{k}{\sqrt{R_{\max}}};g_i(0)\right),~f(p_i)\ge r_0(k,\epsilon).
\end{align}
By Proposition \ref{uniform vol ratio}, $(M,g_i(t),p_i)$ is $\kappa$-noncollapsed. By the estimate (\ref{8-3}) and Proposition \ref{uniform vol ratio}, by taking a subsequence, $(M,g_i(t),p_i)$ converges to a limit $(M_{\infty},g_{\infty}(t),p_{\infty})$ with maximal volume growth. Since the scalar curvature of $(M,g,f)$ decays uniformly, by the same argument as in the proof of Claim \ref{claim-convergence}, we can show that $(M_{\infty},g_{\infty}(t),p_{\infty})$ splits off a line. We assume that $(M_{\infty},g_{\infty}(t))=(N\times\mathbb{R}, g_{N}(t)+ds^2)$. Note that $(\ref{8-3})$ also implies that $(M_{\infty},g_{\infty}(t))$ has bounded curvature. Hence, $g_{N}(t)$ is a three-dimensional $\kappa$-solution with bounded curvature. We also note that $g_{N}(t)$ has maximal volume growth. Hence, $(N,g_N(t))$ must be flat. It follows that $R_{\infty}(p_{\infty},0)=0$. Hence,
\begin{align}
\lim_{i\to\infty}\frac{R(p_i)}{M(p_i)}=R_{\infty}(p_{\infty},0)=0.
\end{align}
This contradicts (\ref{positive-assumption}). Hence, we have completed the proof.

\end{proof}

 \begin{lem}\label{lem-small linear decay constant}
  Let $(M,g,f)$ be a $4$-dimensional steady gradient Ricci soliton with uniform scalar curvature decay which satisfies condition \eqref{Ricci nonnegative outside K}. If $(M,g,f)$ has maximal volume growth, then there exists a constant $r_0(\epsilon)$ such that
  \begin{align}
  R(x)\le \frac{4\epsilon}{f(x)}\quad\forall~f(x)\ge r_0(\epsilon).
  \end{align}
 \end{lem}

 \begin{proof}
 We only need to consider the case that $(M,g,f)$ is non-Ricci-flat.
 Let $\{p_i\}_{i\in \mathbb{N}_{+}}$ be a sequence points such that $f(p_i)=2^i$ and
 \begin{align}
 R(p_i)=\sup_{x\in \Sigma_{f(p_i)}}R(x).
 \end{align}
 We first show that
 \begin{align}\label{estimate for special points}
 R(p_i)\le \frac{2\epsilon}{f(p_i)}\quad\forall~i\ge i_0,
 \end{align}
 for some constant $i_0>0$. We need the following claim.

 \begin{claim}\label{claim-a}
 There are infinitely many $i$ such that
 \begin{align}
 R(p_i)\le \frac{\epsilon}{f(p_i)}.
 \end{align}
 \end{claim}
 \begin{proof}
 We prove the claim by contradiction. If the claim is not true, then there exists a constant $i_0>0$ such that
 \begin{align}
 R(p_i)\ge\frac{\epsilon}{f(p_i)}\quad\forall~i\ge i_0.
 \end{align}
 Let $\varepsilon=\frac{\epsilon}{4}$. We can make $i_0$ large enough such that
 \begin{align}
 f(p_i)\ge r_0(\varepsilon,\epsilon)\quad\forall~i\ge i_0,\notag
 \end{align}
 where $r_0(\varepsilon,\epsilon)$ is the constant in Lemma \ref{lem-curvature fast decay}. By Lemma \ref{lem-curvature fast decay}, we have
 \begin{align}
 \frac{R(p_{i+1})}{R(p_i)}\le \varepsilon\quad\forall~i\ge i_0.\notag
 \end{align}
 Hence,
 \begin{align}\label{8-a}
 R(p_i)\le \varepsilon^{i-i_0}R(p_{i_0})\quad\forall~i>i_0.
 \end{align}
 On the hand, by our assumption, we have
 \begin{align}\label{8-b}
 R(p_i)\ge \frac{\epsilon}{f(p_i)}=\frac{\epsilon}{2^i}\quad\forall~i>i_0.
 \end{align}
 Note that $\varepsilon=\frac{\epsilon}{4}\le \frac{1}{4}$. When $i$ is large enough, we get a contradiction by combining (\ref{8-a}) and (\ref{8-b}).
 \end{proof}
 By Claim \ref{claim-a}, there exists a constant  $i_0$ such that
 \begin{align}
 R(p_{i_0})\le \frac{\epsilon}{f(p_{i_0})}
 \end{align}
 and
 \begin{align}
 f(p_{i_0})\ge r_0(\varepsilon,\epsilon),
 \end{align}
 where $r_0(\varepsilon,\epsilon)$ is the constant in Lemma \ref{lem-curvature fast decay} and we take $\varepsilon=\frac{\epsilon}{4}\le \frac{1}{4}$. If there are only finitely many $p_i$ such that
 \begin{align}\label{decay-7}
     R(p_i)>\frac{\epsilon}{f(p_i)},
 \end{align}
 then  $(\ref{estimate for special points})$ holds for $i\ge i_0+1$, where $i_0$ is the largest number such that $p_i$ satisfies (\ref{decay-7}). Now, we assume that there are infinitely many $N_i$.

 We need the following claim.

 \begin{claim}\label{claim-b}
 Suppose $N>i_0$ and
 \begin{align}
 R(p_{N-1})\le\frac{\epsilon}{f(p_{N-1})},~R(p_{N})>\frac{\epsilon}{f(p_{N})}.
 \end{align}
 Then, we have
 \begin{align}\label{induction-1}
 R(p_{N})\le\frac{2\epsilon}{f(p_{N})}
 \end{align}
 and
 \begin{align}\label{induction-2}
 R(p_{N+1})\le\frac{\epsilon}{f(p_{N+1})}.
 \end{align}
 \end{claim}
 \begin{proof}
 By Lemma \ref{lem-monotone of S}, we have
 \begin{align}
 R(p_{N})\le R(p_{N-1})\le\frac{2\epsilon}{f(p_N)}.
 \end{align}
 We have proved (\ref{induction-1}).

 By Lemma \ref{lem-curvature fast decay} and (\ref{induction-1}), we have
 \begin{align}
 R(p_{N+1})\le\varepsilon R(p_N)\le\frac{2\epsilon\varepsilon}{f(p_N)}=\frac{\epsilon}{f(p_{N+1})}.
 \end{align}
  This completes the proof of the claim.
 \end{proof}

 Now, we let $N_i$ be the number such that $p_{N_i}$ is the $i$-th point such that $N_i>i_0$
 and
 \begin{align}
 R(p_{N_i})>\frac{\epsilon}{f(p_{N_i})}.
 \end{align}
 By Claim \ref{claim-b} and an induction argument, it is easy to see that
 $N_{i+1}\ge N_i+1$ and therefore
 \begin{align}\label{estimate for bad points}
 R(p_{N_i})\le \frac{2\epsilon}{f(p_{N_i})} \quad  \forall~i\ge1.
 \end{align}
 By our definition of $N_i$ and the estimate (\ref{estimate for bad points}), we have completed the proof of (\ref{estimate for special points}).

 For any $x\in M$ such that $f(x)\ge 2^{i_0}$, we assume that $f(x)\in[2^i,2^{i+1})$. By Lemma \ref{lem-monotone of S}  and (\ref{estimate for special points}), we get
 \begin{align}
 R(x)\le R(p_i)\le \frac{2\epsilon}{f(p_i)}=\frac{2\epsilon}{2^i}<\frac{4\epsilon}{f(x)} \quad  \forall~f(x)\ge 2^{i_0}.
 \end{align}
 \end{proof}

 By the curvature estimate in \cite{Chan}, we see that $|{\rm Rm}|(x)\le C R(x)$, for all $x\in M$ and some constant $C$. By Theorem 6.1 in \cite{DZ6} and Theorem \ref{theo-linear of f}, we have the following theorem.
 \begin{theo}\label{theo-linear decay lower bound}
 Let $(M,g,f)$ be a $4$-dimensional $\kappa$-noncollapsed steady gradient Ricci soliton which satisfies condition \eqref{Ricci nonnegative outside K} and the scalar curvature $R(x)$ satisfies
  \begin{align}
  R(x)\le \frac{C}{\rho(x)}\quad \forall~\rho(x)\ge r_0.
  \end{align}
 If $(M,g,f)$ is non-Ricci-flat, then there exist constants $c_1, r_0>0$ such that
 \begin{align}
R(x)\ge\frac{c_1}{\rho(x)}\quad \forall~ \rho(x)\ge  r_0.
\end{align}
\end{theo}

 Now, we prove Theorem \ref{theo-maximal volume growth}.

 \begin{proof}[Proof of Theorem \ref{theo-maximal volume growth}]
 We prove this by contradiction. Suppose $(M,g,f)$ is not Ricci-flat.

 \textbf{Case 1:} If the scalar curvature does not have uniform decay, then $(M,g,f)$ dimension reduces to a $3$-dimensional $\kappa$-noncollapsed and non-flat steady Ricci soliton with maximal volume growth by Theorem \ref{theo-dimension reduce without decay} (which is proved in Section \ref{section-dimension reduction}). However,  $3$-dimensional non-flat ancient $\kappa$-solutions cannot have maximal  volume growth (see \cite{Pe1}). Hence, the steady Ricci soliton must be Ricci flat.

 \textbf{Case 2:} In this case, we assume that the scalar curvature has uniform decay. By Lemma \ref{lem-small linear decay constant}, there exists a constant $r_0$ such that
 \begin{align}
 R(x)\le \frac{1}{f(x)} \quad \forall~f(x)\ge r_0.\notag
 \end{align}
By Proposition \ref{uniform vol ratio}, we see that $(M,g,f)$ is $\kappa$-noncollapsed.  Hence, $(M,g,f)$ is a $\kappa$-noncollapsed steady gradient Ricci soliton with linear curvature decay. By Theorem \ref{theo-linear decay lower bound} and Theorem \ref{theo-linear of f}, there exist constants $c_1, r_1>0$ such that
 \begin{align}
R(x)\ge\frac{c_1}{f(x)} \quad  \forall~ f(x)\ge  r_1.\notag
\end{align}
On the other hand, by taking $\epsilon=\frac{c_1}{2}$ in Lemma \ref{lem-small linear decay constant}, we have
\begin{align}
R(x)\le \frac{c_1}{2f(x)} \quad \forall~f(x)\ge r_0\left(\frac{c_1}{2}\right).\notag
\end{align}
Hence, we get a contradiction. We have completed the proof.

 \end{proof}

\section{Noncollapsed GRS with nonnegtive Ricci curvature}\label{section-compactness theorem on solitons}

In this section, we will deal with $4$-dimensional $\kappa$-noncollapsed steady gradient Ricci solitons with uniform scalar curvature decay and satisfying condition (\ref{Ricci nonnegative outside K}).
\begin{prop}\label{prop-section 7}
Let $(M,g,f)$ be a $4$-dimensional steady gradient Ricci soliton which is not Ricci flat and  satisfies the hypotheses of Theorem \ref{theo-nonnegative Ricci}.
Given any $\gamma>0$, if $p_i\in M$ and $r_i>0$ satisfy ${\rm Vol} \,B(p_i,r_i;g)\ge \gamma r_i^4$, then there exists  a constant $C(\gamma)$ such that $r_i^2R(q)\le C(\gamma)$ for all $q\in B(p_i,r_i;g)$, where $C(\gamma)$ is independent of $p_i$ and $r_i$.
\end{prop}

\begin{proof}
We prove the proposition by contradiction. If the proposition is not true, then by taking a subsequence, we may assume that
\begin{align}
{\rm vol} B(p_i,r_i;g)\ge \gamma r_i^4,\\
r_i^2Q_i\to \infty,~ \mbox{as} ~i\to\infty,\label{7-1}
\end{align}
where $Q_i=R(q_i)$ and $q_i\in B(p_i,r_i;g)$.

By 
Lemma 9.37 
in \cite{MT}, we can find points $q_i^{\prime}\in B(p_i,2r_i;g)$ and constants $s_i\le r_i$ such that
\begin{align}\label{7-2}
R(q_i^{\prime})s_i^2=Q_ir_i^2,
\end{align}
 and
 \begin{align}\label{7-3}
 R(q)\le 4R(q_i^{\prime})\,\text{ for all }\,q\in B(q_i^{\prime},s_i;g).
 \end{align}
We set $Q_i^{\prime}=R(q_i^{\prime})$. Then, (\ref{7-1}) and (\ref{7-2}) imply that
\begin{align}\label{7-4}
s_i^2Q_i^{\prime}\to \infty,~ \mbox{as} ~i\to\infty.
\end{align}
We also note that
\begin{align*}
B(p_i,r_i;g)\subseteq B(q_i^{\prime},3r_i;g).
\end{align*}
It follows that
\begin{align}\label{volume growth of q i prime}
{\rm vol} \, B(q_i^{\prime},3r_i;g)\ge {\rm vol} \, B(p_i,r_i;g)\ge \gamma r_i^4=\frac{\gamma}{81}\cdot (3r_i)^4.
\end{align}

\textbf{Step 1:} We show that $\{q_{i}^{\prime}\}_{i\in \mathbb{N}}$  tends to infinity. If it is not true, then there exists a subsequence of $\{q_{i}^{\prime}\}_{i\in \mathbb{N}}$ that stays in a bounded subset of $M$. By taking a subsequence, we may assume that $q_{i}^{\prime}\to q_{\infty}^{\prime}$ as $i\to\infty$.  Now, we consider the volume growth of $(M,g)$. Let $l_i=3r_i$. By (\ref{7-4}), we have
\begin{align}
R(q_i^{\prime})\cdot \frac{s_i^2}{r_i^2}\cdot(\frac{l_i}{3})^2=Q_i^{\prime}s_i^2\to\infty.
\end{align}
Note that $R(q_i^{\prime})\le R_{\max}$ and $s_i/r_i\le 1$. Hence, we get $l_i\to\infty$. Let $\varepsilon_i=d_g(q_i^{\prime},q_{\infty}^{\prime})$. Then, $\varepsilon_i\to0$. For $i$ large, by (\ref{volume growth of q i prime}), we have
\begin{align*}
{\rm vol} \, B(q_{\infty}^{\prime},l_i+\varepsilon_i;g)&\ge {\rm vol} \, B(q_i^{\prime},l_i;g)\ge \frac{\gamma}{162}\cdot(l_i+\varepsilon_i)^4.
\end{align*}
As $l_i+\varepsilon_i$ tends to infinity, the inequality above implies that $(M,g)$ has maximal volume growth. Therefore, $(M,g)$ is Ricci flat by Theorem \ref{theo-maximal volume growth}. However, $(M,g)$ is not Ricci flat by our assumption. Hence, $\{q_{i}^{\prime}\}_{i\in \mathbb{N}}$  tends to infinity.

\textbf{Step 2:} We show that $B(q_i^{\prime},3r_i;g)\cap K=\varnothing$ when $i$ large. If it is not true, we may assume that $B(q_i^{\prime},3r_i;g)\cap K\neq\varnothing$ by taking a subsequence. Let $d={\rm Diam}(K,g)$ and fix a point $o\in K$. Let $l_i=3r_i$ as in \textbf{Step 1}. Hence, $d_g(q_i^{\prime},o)\le d+l_i$. Then, $B(q_i^{\prime},l_i;g)\subseteq B(o,2(d+l_i);g)$. By (\ref{volume growth of q i prime}), for $i$ large, we have
\begin{align}\label{volume growth around o}
\frac{{\rm vol} \, B(o,2(d+l_i);g)}{[2(l_i+d)]^4}\ge \frac{\gamma}{81}\cdot\frac{l_i^4}{16(l_i+d)^4}\ge\frac{\gamma}{2\cdot 6^4} .
\end{align}
As in \textbf{Step 1}, $l_i\to\infty$. Hence, (\ref{volume growth around o}) implies that $(M,g)$ has maximal volume growth. Hence, $(M,g)$ is Ricci flat. This contradicts our assumption. Hence, $B(q_i^{\prime},3r_i;g)\cap K=\varnothing$ when $i$ is sufficiently large.

\textbf{Step 3:} Let $\phi_t$ be generated by $-\nabla f$ and $g(t)=\phi_{t}^{\ast}g$. Let $g_i(t)=Q_i^{\prime} g((Q_i^{\prime})^{-1}t)$. Since $B(q_i^{\prime},3r_i;g)\cap K=\varnothing$ for large $i$, we have $\phi_t(x)\in M\setminus K$, for all $x\in B(q_i^{\prime},3r_i;g)$. Hence, ${\rm Ric}(x,t)\ge0$ for all $x\in B(q_i^{\prime},3r_i;g)$.

By the Bishop-Gromov volume comparison theorem and (\ref{volume growth of q i prime}),
\begin{align}\label{7-5}
{\rm vol} \, B(q_i^{\prime},s,g)\ge \frac{\gamma}{81}\cdot s^4\,\text{ for all }\,s\le s_i.
\end{align}
Since the Ricci curvature is nonnegative, we have $\frac{\partial}{\partial t}R(x,t)\ge 0$, for all $x\in B(q_i^{\prime},3r_i;g)$.  Then, by (\ref{7-3}), we have
\begin{align}
R_{g_i(t)}(q)\le R_{g_i(0)}(q)\le4\,\text{ for all }\,t\le 0,~q\in B(q_i^{\prime},s_i;g).\notag
\end{align}
By \cite{Chan}, there exists a constant $C$ such that
 \begin{align}\label{7-6}
|{\rm Rm}_{g_{i}(t)}(q)|\le C R_{g_i(t)}(q)\le 4C\,\text{ for all }\,t\le 0,~q\in B(q_i^{\prime},s_i;g)
\end{align}
Note that $B(q_i^{\prime},s_i\sqrt{Q_i^{\prime}};g_i(0))=B(q_i^{\prime},s_i;g)$. By (\ref{7-4}), (\ref{7-5}) and (\ref{7-6}), we obtain that $(M,g_i(t),q_i^{\prime})$ converges subsequentially to a complete limit $(M_{\infty},g_{\infty}(t),q_{\infty})$ for $t\le0$. Moreover, (\ref{7-4}) and (\ref{7-5}) implies that the asymptotic volume ratio of $(M_{\infty}, g_{\infty}(0))$ is greater than $\frac{\gamma}{81}$.

Note that $q_{i}^{\prime}$ tends to infinity. Then, $R(q_{i}^{\prime})\to0$ as $i\to\infty$ by assumption. Hence, we see that the limit $(M_{\infty},g_{\infty}(t),q_{\infty})$ splits off a line. Hence, $(M_{\infty},g_{\infty}(t),q_{\infty})$ is the product of a line and a $3$-dimensional $\kappa$-solution. Since the asymptotic volume ratio of any $3$-dimensional $\kappa$-solution is zero, the asymptotic volume ratio of $(M_{\infty},g_{\infty}(t),q_{\infty})$ must be zero, too.  This contradicts the volume growth of $(M_{\infty}, g_{\infty}(0))$  that we have obtained.

The proof of the proposition is complete.

\end{proof}

The following lemma is similar to Corollary 2.4 in \cite{DZ2}.

\begin{lem}\label{lem-ASR}
Let $(M,g,f)$ be a steady gradient Ricci soliton satisfying the condition in Proposition \ref{prop-section 7}. Then, its
asymptotic scalar curvature ratio
$\mathcal{R}(M,g)= \limsup_{x\rightarrow\infty} R(x)d(x,x_0)^2=\infty$.
\end{lem}

\begin{proof}
We prove the corollary by contradiction. Suppose $\mathcal{R}(M,g)<A$ for some positive constant $A>1$. For a fixed point $p\in M$,
we have $R(x)\leq Ar^{-2}$ for all $x\in M\setminus B(p,r)$ when $r>r_{0}$. Fix any $q\in B(p, 3\sqrt{A}r)\setminus B(p, 2\sqrt{A}r)$. Then, we have $R(x)\leq r^{-2}$ for all $x\in B(q,r)$.
Since $(M,g)$ is $\kappa$-noncollapsed, we get ${\rm vol} \, B(q,r)\geq \kappa r^{4}$. Hence,
\begin{align}{\rm vol}\,B(p, (3\sqrt{A}+1)r) &\ge {\rm vol}\,B(q, r)\notag\\
&\ge \kappa(3\sqrt{A}+1)^{-n}(3\sqrt{A}+1)r)^{n}\,\text{ for all }\,r>r_{0}.\notag
\end{align}
It follows that
$$\mathcal{V}(M,g)\geq \kappa(3\sqrt{A}+1)^{-n}.$$
This contradicts our assumption that $\mathcal{V}(M,g)=0$.
\end{proof}

Now, we begin to prove Theorem \ref{theo-compactness}.
The proof of Theorem \ref{theo-compactness} is similar to the proof of the compactness theorem of $3$-dimensional $\kappa$-solutions in \cite{Pe1} (see also \cite{MT}). Under the hypotheses of Theorem \ref{theo-compactness}, we  let $g_i(t)=R(p_i)g(R^{-1}(p_i)t)$.  By Lemma \ref{lem-ASR}, we can always find $q_i$ such that
\begin{align}\label{assumption-r}
d_{g_i(0)}(p_i,q_i)^{2}R_{g_i(0)}(q_i)=1.
\end{align}
We first note that the following lemma holds.

\begin{lem}\label{lem-q-i infinity}
$q_i$ tends to infinity.
\end{lem}
\begin{proof}
If the lemma is not true, then we may assume that the $q_i$ converge to a point $q_{\infty}$. Since $(M,g)$ is not Ricci flat, $R(q_{\infty})>0$. Note that (\ref{assumption-r}) implies that
\begin{align}\label{dist of p-i q-i}
d_g(p_i,q_i)^2R(q_i)=1.
\end{align}
By assumption, $p_i$ tends to infinity. Therefore, $d_g(p_i,q_i)\to\infty$ as $i\to\infty$.  Hence, $R(q_i)\to 0$ as $i\to\infty$ by (\ref{dist of p-i q-i}). This contradicts the fact that the $q_i$ converge to $q_{\infty}$ and $R(q_{\infty})>0$.
\end{proof}

Let $d_i=d_{g_i(0)}(p_i,q_i)$. Then, we have the following curvature estimate.

\begin{lem}\label{lem-local bound}
There is a uniform constant $C>0$  such that $R_{g_i(0)}(x)\leq C R_{g_{i}(0)}(q_i)$ for all $x\in B(q_i,2d_i;g_i(0))$.
\end{lem}

\begin{proof}
Suppose that the lemma is not true. By taking a subsequence, we may assume that there exist points $q_{i}^{\prime}\in B(q_{i},2d_{i}; g_i(0))$ such that $$\lim_{i\rightarrow\infty}(2d_{i})^{2}R(q_{i}^{\prime},0)=\infty.$$
 By Proposition \ref{prop-section 7}, for any $\gamma>0$, there is an $i(\gamma)$ such that
$${\rm vol}\,B(q_{i},2d_{i}, 0)<\gamma(2d_{i})^{4}\,\text{ for all }\,i>i(\gamma).$$
Hence, by
applying the diagonal
method, we may assume that
\begin{equation}\label{eq:1}
\lim_{i\rightarrow\infty}{\rm vol}\,B(q_{i},2d_{i}, g_i(0))/(2d_{i})^{4}=0.
\end{equation}
In particular,
$${\rm vol}\, B(q_{i},2d_{i};g_i(0))<(\omega/2)(2d_{i})^{4}\,\text{ for all }\,i\ge i_0,$$
where $\omega$ is the volume of unit ball in $\mathbb{R}^{4}$ and $i_0$ is a constant.

Let $F_i(s)=\frac{{\rm vol}\, B(q_{i},s, g_i(0))}{s^4}$, for $s\in(0,2d_i]$. Note that $F_i(s)$ is continuous. Moreover,
\begin{align*}
\lim_{s\to0}F_i(s)=\omega\quad{\rm and}\quad F_i(2d_i)<\frac{\omega}{2}.
\end{align*}
Therefore, there exists an $r_{i}<2d_{i}$ for each $i\in \mathbb{N}$ such that $F_i(r_i)=\frac{\omega}{2}$, i.e.,
\begin{equation}\label{eq:2}
{\rm vol}\,B(q_{i},r_{i}, g_i(0))=(\omega/2)r_{i}^{4}.
\end{equation}
By (\ref{eq:1}) and (\ref{eq:2}) we have
\begin{align}\label{eq:3}
\lim_{i\rightarrow\infty}r_{i}/d_{i}=0.
\end{align}

Next we consider the sequence of rescaled ancient flows $(M_{i},g^{\prime}_{i}(t), q_{i})$, where $g^{\prime}_{i}(t)=r_{i}^{-2}g_{i}(r_{i}^{2}t)$. Since we want to use Proposition \ref{prop-section 7} to get the curvature estimates on geodesic balls of $(M_{i},g^{\prime}_{i}(t))$, we need to show that $B(q_{i},A; g^{\prime}_{i}(0))\cap K=\varnothing$ for any fixed $A\gg1$. It suffices to exclude the following two cases.

\textbf{Case 1:} There exist infinite many $i$ such that $B(q_{i},A; g^{\prime}_{i}(0))\cap K\neq\varnothing$ and $r_i\sqrt{R^{-1}(p_i)}$ is uniformly bounded. In this case, note that
\begin{align}\label{geodesic ball rescale}
B(q_i,Ar_i\sqrt{R^{-1}(p_i)};g)=B(q_{i},A; g^{\prime}_{i}(0)).
\end{align}
Then,
\begin{align}\label{intersection not empty}
B(q_i,Ar_i\sqrt{R^{-1}(p_i)};g)\cap K\neq\varnothing .
\end{align}
Suppose $r_i\sqrt{R^{-1}(p_i)}\le C$ for all $i$. Then, (\ref{intersection not empty}) implies that $d_g(q_i,K)\le AC$. Hence, $q_i$ stays in a bounded set of $M$ as $i\to\infty$. This contradicts Lemma \ref{lem-q-i infinity}. So this case is impossible.

\textbf{Case 2:} There exist infinitely many $i$ such that $B(q_{i},A; g^{\prime}_{i}(0))\cap K\neq\varnothing$ and $r_i\sqrt{R^{-1}(p_i)}\to \infty$ for $i\to\infty$. In this case, let $l_i=r_i\sqrt{R^{-1}(p_i)}$. Note that (\ref{geodesic ball rescale}) and (\ref{intersection not empty}) still hold. (\ref{intersection not empty}) implies that $d_g(q_i,K)\le Al_i$. Fix a point $o\in K$. We have
\begin{align}\label{balls contain}
B(q_i,Al_i;g)\subseteq B(o, Al_i+d;g),
\end{align}
 where $d={\rm Diam}(K,g)$. By (\ref{balls contain}) and (\ref{eq:2}), we have
 \begin{align}\label{volume growth-2}
 {\rm vol}\,B(o,Al_i+d;g)\ge {\rm vol}\,B(q_i,l_i;g)=\frac{\omega}{2}\cdot l_i^4.
 \end{align}
 Since $l_i\to\infty$ and $A,d$ are constants,  (\ref{volume growth-2}) implies that $(M,g)$ has maximal volume growth. This is impossible by Theorem \ref{theo-maximal volume growth}.

 \textbf{Case 1} and \textbf{Case 2} imply that for any $A\gg1$, there exists an $i(A)>0$ such that for any $i\ge i(A)$, we have
 \begin{align}
 B(q_i,A;g^{\prime}_i(0))\cap K=\varnothing.
 \end{align}
Hence, ${\rm Ric}(x,t)\ge 0$ for all $x\in B(q_i,A;g^{\prime}_i(0))$ and $t\le0$ when $i\ge i(A)$.

 By (\ref{eq:2}), we have
$${\rm vol}\,B(q_{i},A; g^{\prime}_{i}(0))\geq {\rm vol}\,B(q_{i},1, g^{\prime}_{i}(0))=\frac{\omega}{2A^{4}}\cdot A^{4},$$
where $A>0$ is any fixed constant. It follows that
\begin{align*}
{\rm vol}\, B(q_i,Al_i;g)\ge \frac{\omega}{2A^4}\cdot(Al_i)^4,
\end{align*}
where $l_i=r_i\sqrt{R^{-1}(p_i)}$ and $i\ge i(A)$. Note that $$B (q_{i},A; g^{\prime}_{i}(0))=B(q_i,Al_i;g).$$
By applying Proposition \ref{prop-section 7} to the ball $B(q_i,Al_i;g)$,
there is a constant $K(A)$ independent of $i$ such that
$$A^{2}R_{g_{i}^{\prime}(0)}(q)=(Al_i)^2R(q)\leq K(A),\,\forall\,q\in B (q_{i},A; g^{\prime}_{i}(0)).$$
Since the flow $g(t)$ is generated by a steady gradient Ricci soliton and the Ricci curvature is nonnegative on $B(q_i,A;g^{\prime}_i(0))$, its scalar curvature is non-decreasing in $t$. Hence,
the scalar curvature on $B(q_{i},A;g^{\prime}_{i}(0))\times(-\infty,0]$ is uniformly bounded by $K(A)/A^2$. By \cite{Chan}, there exists a constant $C$ such that
 \begin{align}\label{scalar control curvature}
 |{\rm Rm}|\le CR(x)\,\text{ for all }\,x\in M.
 \end{align}
 It follows that
  \begin{align}\label{eq:4}
 |{\rm Rm}_{g_i(t)}|_{g_{i}(t)}(x)\le CR_{g_{i}(t)}(x)\le CK(A)\,\text{ for all }\,(x,t)\in B(q_{i},A;g^{\prime}_{i}(0))\times(-\infty,0].
 \end{align}
 By Hamilton's Cheeger--Gromov compactness theorem, $(M_{i},g^{\prime}_{i}(t), q_{i})$ converges to a limit flow $(M_{\infty},g_{\infty}(t),q_{\infty})$.
Note by (\ref{eq:3}) that
$$R(q_{\infty}, g_{\infty}(0))=\lim_{i\rightarrow\infty}R(q_{i}, g_{i}'(0))=\lim_{i\rightarrow\infty}\frac{(r_{i})^{2}}{d_{i}^{2}}=0.$$
Therefore, the strong maximum principle implies that $(M_{\infty},g_{\infty}(t))$ is a Ricci-flat flow. (\ref{eq:4}) implies that
\begin{align}
|{\rm Rm}_{g_{\infty}(t)}|_{g_{\infty}(t)}(x)\le CR_{g_{\infty}(t)}(x)\,\text{ for all }\,x\in M_{\infty}.
\end{align}
Hence, $(M_{\infty},g_{\infty}(t))$ is flat.

At last, we prove that $(M_{\infty},g_{\infty}(t))$ is isometric to Euclidean space for any $t\le0$. Fix any $r>0$. Obviously,
$$\sup_{x\in B(q_{\infty},r; g_{\infty}(0))}|{\rm Rm}(x)|=0\le \varepsilon,$$
where $\varepsilon$ can be chosen so that
$\frac{\pi}{\sqrt{\varepsilon}}>2r.$
Note that $(M_{\infty},g_{\infty}(t))$ is $\kappa$-noncollapsed for each $t\le 0$. Thus we have
$${\rm vol}\,B (q_{\infty},r; g_{\infty}(0))\geq \kappa r^{4}.$$
It follows from the estimate of Cheeger, Gromov, and Taylor \cite{CGT} that
$${\rm inj}(q_{\infty})\geq \frac{\pi}{2\sqrt{\varepsilon}}\frac{1}{1+\frac{\omega(r/4)^{4}}{{\rm vol}(B (q_{\infty},r/4; g_{\infty}(0) ))}}\geq \frac{\kappa}{\kappa+\omega}\cdot r.$$
Hence $B(q_{\infty},\frac{\kappa}{\kappa+\omega}\cdot r; g_{\infty}(0))$ is simply connected for all $r>0$. Therefore, $M_{\infty}$ is simply connected, and consequently $g_\infty(t)$ are all isometric to the Euclidean metric.

 Since $(M_{\infty},g_{\infty}(t))$ is isometric to $4$-dimensional Euclidean space, we obtain that ${\rm vol}(B(q_{\infty},1; g_{\infty}(0)))=\omega$. On the other hand, by the convergence of $(M_{i},g^{\prime}_{i}(t);p_{i})$ and the relation (\ref{eq:2}), we get
$${\rm vol}(B (q_{\infty},1; g_{\infty}(0)))=\omega/2.$$
This is a contradiction.
\end{proof}

We still need to show that the $B(q_i,2d_i;g_i(0))$ stay outside of $K$ as $i\to\infty$.
\begin{lem}\label{lem-stay outside K}
For $A\gg 2$, there exists a constant $i(A)>0$ such that $B(q_i,Ad_i;g_i(0))\cap K=\varnothing$ for $i\ge i(A)$.
\end{lem}
\begin{proof}
Let $h_i(t)=d_i^{-2}g_i(d_i^2t)$. Then, $h_i(0)=R(q_i)g$. So, we only need to show that $B(q_i,A\sqrt{R^{-1}(q_i)};g)\cap K=\varnothing$ when $i\ge i(A)$ for some $i(A)>0$.
If it is not true, then we may assume that $B(q_i,A\sqrt{R^{-1}(q_i)};g)\cap K\neq\varnothing$ as $i\to\infty$ by taking a subsequence. Let $l_i=A\sqrt{R^{-1}(q_i)}$. By Lemma \ref{lem-q-i infinity} and the curvature uniform decay assumption, we have $l_i\to\infty$ as $i\to\infty$. Since $h_i(t)$ is $\kappa$-noncollapsing, by Lemma \ref{lem-local bound} and \cite{Chan}, we have
\begin{align}
{\rm vol} \, B(q_i,l_i;g)\ge c\kappa l_i^4,\notag
\end{align}
where $c$ is a positive constant. Let $d={\rm Diam}(K,g)$. Fixing $o\in K$, we have
\begin{align}
{\rm vol} \, B(o,2(Al_i+d);g)\ge {\rm vol} \, B(q_i,l_i;g)\ge c\kappa l_i^4.\notag
\end{align}
Note that $l_i\to\infty$ as $i\to\infty$. Since $A$ and $d$ are constants, $(M,g)$ has maximal volume growth. This contradicts Theorem \ref{theo-maximal volume growth}.
\end{proof}

Compared with $3$-dimensional ancient $\kappa$-solutions, we do not have the Harnack inequality (see \cite{H3}) in our case. Fortunately, we have the following lemma.

\begin{lem}\label{lem-positive lower bound}
Let $(M_i,g_i(t),p_i)$ be a sequence of $\kappa$-noncollapsed ancient and complete Ricci flows. Suppose there exist a constant $C$ such that
\begin{align}
|{\rm Rm}_{g_i(t)}(x)|_{g_{i}(t)}\le CR_{g_i(t)}(x)\,\text{ for all }\,x\in M_i,~t\le 0
\end{align}
and
\begin{align}
\frac{\partial}{\partial t}R_{g_i(t)}(x)\ge0\,\text{ for all }\,x\in M_i,~t\le 0.
\end{align}
We also assume that $R_{g_{i}(0)}(p_i)=1$ and
 \begin{align}
 R_{g_{i}(0)}(x)\le C_1\,\text{ for all }\,x\in B(p_i,2;g_{i}(0)).
 \end{align}
 Then, there exists a constant $\delta>0$ independent of $i$ and $x$ such that
 \begin{align}
 R_{g_i(0)}(x)\ge \delta\,\text{ for all }\,i\in \mathbb{N}~\mbox{and}~d_{g_i(0)}(x,p_i)=1.
 \end{align}
\end{lem}
\begin{proof}
We prove this by contradiction. If the lemma is not true, we can find $x_i\in M_i$ such that $d_{g_i(0)}(x_i,p_i)=1$ and
\begin{align}\label{limit is zero}
R_{g_i(0)}(x_i)\to0,~\mbox{as}~i\to\infty.
\end{align}
By our assumption, it is easy to see that  $(B(p_i,2;g_{i}(0)),g_i(t),p_i)$ converge subsequentially to a limit $(B_{\infty}, g_{\infty}(t),p_{\infty})$. Note that $R_{g_{\infty}(0)}(p_{\infty})=1$. Since $g_i(t)$ is ancient and complete for each $i$, $R_{g_i(t)}(x)\ge0$ (See \cite{Ch}). Hence, $R_{g_{\infty}(t)}(x)\ge0$ for all $x\in B_{\infty}$ and $t\le 0$.  By (\ref{limit is zero}), we have $R_{g_{\infty}(0)}(x_{\infty})=0$, where $x_{\infty}$ is the limit of $x_i$. By the maximum principle, we see that $R_{g_{\infty}(t)}(x)$ is flat for all $x\in B_{\infty}$ and $t\le0$. This contradicts the fact that $R_{g_{\infty}(0)}(p_{\infty})=1$. Hence, we have completed the proof.
\end{proof}

Now, we complete the proof of Theorem \ref{theo-compactness}.

\begin{proof}[Proof of Theorem \ref{theo-compactness}]
By Lemma \ref{lem-local bound}, there is a uniform constant $C>0$  such that $R_{g_i(0)}(x)\leq C R_{g_{i}(0)}(q_i)$ for all $x\in B(q_i,2d_i;g_i(0))$, where $d_i^2R_{g_{i}(0)}(q_i)=1$. Let $h_{i}(t)=d_i^{-2}g_i(d_i^2t)$, for $t\le 0$. Then, $R_{h_{i}(0)}(q_i)=1$ and
 \begin{align}
 R_{h_{i}(0)}(x)\le C\,\text{ for all }\,x\in B(q_i,2;h_{i}(0))=B(q_i,2d_i;g_i(0)).
 \end{align}
 We also note that $d_{h_i(0)}(p_i,q_i)=1$. By Lemma \ref{lem-stay outside K}, ${\rm Ric}_{h_i(t)}(x)\ge0$ for all $x\in B(q_i,2;h_{i}(0))$ and $t\le 0$. Hence, $R_{h_i(t)}(x)$ is increasing in $t$ for all $x\in B(q_i,2;h_{i}(0))$.
 Applying Lemma \ref{lem-positive lower bound} to $(M,h_i(t),q_i)$ and $p_i$, there exists a positive constant $\delta>0$ such that
 \begin{align}
 R_{h_i(0)}(p_i)\ge \delta.
 \end{align}
 It follows that
 \begin{align}
 d_i^2=\frac{1}{R_{g_i(0)}(q_i)}=\frac{R_{g_i(0)}(p_i)}{R_{g_i(0)}(q_i)}=R_{h_i(0)}(p_i)\ge \delta.
 \end{align}
 Combining the above result and Lemma \ref{lem-local bound}, we have proved the following estimate
 \begin{align}
 R_{g_i(0)}(x)\leq C\delta^{-1}\,\text{ for all }\,x\in B(p_i,\sqrt{\delta};g_i(0)).
 \end{align}
 We take $\varepsilon=\frac{1}{\sqrt{C+1}}$. By the $\kappa$-noncollapsing property of $g_i(t)$, we get
$${\rm vol} \, B(p_{i},\varepsilon;g_i(0))\geq \kappa \varepsilon^{4}.$$

For any $r>0$ large enough, $B(p_{i},\varepsilon+r;g_i(0))\subseteq B(q_i,2r\delta^{-\frac{1}{2}}d_i;g_i(0))$. By Lemma \ref{lem-stay outside K}, $B(p_{i},\varepsilon+r;g_i(0))\cap K=\varnothing$ for $i\ge i(r\delta^{-\frac{1}{2}})$, where $i(r\delta^{-\frac{1}{2}})$ is a constant. Hence, $(M,g_i(t))$ has nonnegative Ricci curvature on $B(p_{i},\varepsilon+r;g_i(0))$ when $i\ge i(r\delta^{-\frac{1}{2}})$.
Therefore, we have
$${\rm vol} \, B(p_{i},\varepsilon+r;g_i(0))\geq {\rm vol} \, B(p_{i},\varepsilon;g_i(0))\geq \frac{\kappa}{(1+(r/\varepsilon))^{4}}(\varepsilon+r)^{4}.$$
Hence,
$${\rm vol} \, B(p_{i},(\varepsilon+r)\sqrt{R^{-1}(p_i)};g)\ge \frac{\kappa}{(1+(r/\varepsilon))^{4}}\cdot [(\varepsilon+r)\sqrt{R^{-1}(p_i)}]^4.$$
 Applying Proposition \ref{prop-section 7} to each ball ${\rm vol} \, B(p_{i},(\varepsilon+r)\sqrt{R^{-1}(p_i)};g)$, we see that there is a $C(r)$ independent of $i$ such that
$$R_{g_i(0)}(q)\leq C(r)(r+\varepsilon)^{-2}\,\text{ for all }\, q\in B(p_{i},\varepsilon+r;g_i(0)).$$
Since the scalar curvature is non-decreasing, we also get
\begin{align}\label{curvature estimate on geodesic bar r+}
    R_{g_i(0)}(q)\leq C(r)(r+\varepsilon)^{-2}\,\text{ for all }\, q\in B(p_{i},\varepsilon+r;g_i(0)).
\end{align}
Recall that $(M,g_i(t))$ has nonnegative Ricci curvature on $B(p_{i},\varepsilon+r;g_i(0))$ and $g_i(t)$ satisfies the Ricci flow equation. Then, 
\begin{align*}
    g_i(q,t)\ge g_i(q,0)\quad\forall~t\le0,~q\in B(p_i,\varepsilon+r;g_i(0)).
\end{align*}
Therefore,
\begin{align*}
 B(p_i,\varepsilon+r;g_i(t))\subset B(p_i,\varepsilon+r;g_i(0))\quad\forall~t\le0.
\end{align*}
By (\ref{curvature estimate on geodesic bar r+}) and (\ref{scalar control curvature}), for $t\le 0$, we have
\begin{align}\label{curvature estimate on geodesic bar r+t}
    |{\rm Rm}|_{g_i(t)}(q)\leq C(n)C(r)(r+\varepsilon)^{-2}\,\text{ for all }\, q\in B(p_{i},\varepsilon+r;g_i(t)).
\end{align}
For any $r>0$ large enough, we can find constant $C(r)$ such that (\ref{curvature estimate on geodesic bar r+t}) holds. Hence, Theorem 1.7 in \cite{To14} implies that $(M,g_{i}(t),p_{i})$ subsequentially converges to a complete Ricci flow $(M_\infty, g_\infty(t))$ for any $t\le 0$. The splitting follows from the scalar curvature decay as proved in Claim \ref{claim-split}. Note that the limit flow is a $3$-dimensional ancient flow which has monotonic scalar curvature in $t$ and distance-curvature estimate. Hence, the flow has uniformly bounded curvature at each time slice.

\end{proof}

 \section{Dimension reduction for steady GRS without curvature decay}\label{section-dimension reduction}

 In this section, we prove Theorem \ref{theo-dimension reduce without decay}.  We first introduce a lemma.

\begin{lem}\label{lem-convergence when R no decay}
 Let $(M,g,f)$ be  an $n$-dimensional $\kappa$-noncollpased steady gradient Ricci soliton  with bounded curvature. Suppose $\{p_i\}_{i\in\mathbb{N}}$ is a sequence of points tending to infinity.  Then, $(M,g,p_i)$ subsequentially converges to a gradient steady Ricci soliton $(M_{\infty},g_{\infty},p_{\infty})$.
 \end{lem}

\begin{proof}
Since $(M,g,f)$ is $\kappa$-noncollpased and has bounded cuvature, it is easy to see that $(M,g,p_i)$ subsequentially converges  in the Cheeger-Gromov sense to a limit $(M_{\infty},g_{\infty},p_{\infty})$. Let $f_i(x)=f(x)-f(p_i)$, for $i\in\mathbb{N}$ and $R_{\max}=\sup_{x\in M}R(x)$. By integrating $f(x)$ along a minimal geodesic connecting $p_i$ and $x$, we get
\begin{align}
|f(x)-f(p_i)|\le \sqrt{R_{\max}}d(x,p_i).\notag
\end{align}
Hence, for any fixed $r>0$, we have
\begin{align}\label{C0-bound}
|f_i(x)|\le r\quad\forall~x\in B(p_i,r;g).
\end{align}
Since the curvature is bounded, we also note that
\begin{align}\label{Ck-bound}
|\nabla^kf_i|(x)=|\nabla^{k-2}{\Ric}|(x)\le C \quad \forall~x\in M.
\end{align}
By (\ref{C0-bound}) and (\ref{Ck-bound}), we see that $f_i(x)$ converges to a smooth function $f_{\infty}(x)$ defined on $M_{\infty}$. Note that
\begin{align}
\nabla\nabla f_i(x)=\nabla\nabla f(x)={\Ric}\quad\forall~x\in M.\notag
\end{align}
By the convergence of $(M,g,p_i)$ and $f_i(x)$, we have
\begin{align}
\nabla\nabla f_{\infty}(x)={{\Ric}}_{\infty}(x)\quad\forall~x\in M_{\infty}.\notag
\end{align}
We have completed the proof.
\end{proof}

 Next, we prove a special case of Theorem \ref{theo-dimension reduce without decay}.
 \begin{lem}\label{lem-splitting-1}
 Let $(M,g,f)$ be  an $n$-dimensional $\kappa$-noncollpased steady gradient Ricci soliton  with bounded curvature and nonnegative Ricci curvature. Suppose the scalar curvature $R(x)$ attain its maximum at $o\in M$ and $|\nabla f|(o)=1$. Then, $(M,g,f)$ weakly dimension reduces to an  $(n-1)$-dimensional steady gradient Ricci soliton.
 \end{lem}
\begin{proof}
Let $\phi_t$ be a group of diffeomorphisms generated by $-\nabla f$ and $g(t)=\phi_t^{\ast}g$. Let $\gamma(s)$ be the integral curve of $\nabla f$ passing through $o$ such that $\gamma(0)=o$. Note that $\gamma(s)=\phi_{-s}(o)$. We first show that $\gamma(s)$ is a geodesic with respect to $g$.

We may assume that $R(\gamma(s))=R_{\max}$ for all $s\in\mathbb{R}$,\footnote{It is easy to check that $R(\gamma(s))=R_{\max}$ for all $s\le 0$. Then, $(M,g, \gamma(-i))$ converges to a limit $(M_{\infty},g_{\infty},p_{\infty})$ with potential $f_{\infty}$ when $i\to\infty$. Let $\varphi_t$ be generated by $-\nabla f_{\infty}$. Then, $R_{\infty}(\varphi_t(p_{\infty}))=R_{\max}$ for all $t\in \mathbb{R}$. Hence, one may replace $(M,g,f,o)$ by $(M_{\infty},g_{\infty},f_{\infty},p_{\infty})$.} where $R_{\max}=\sup_{x\in M}R(x)$. Hence,
\begin{align}
{\Ric}(\nabla f, \nabla f)(\gamma(s))=-\frac{1}{2}\cdot\frac{{\rm d}R(\gamma(s))}{{\rm d}s}=0.
\end{align}
Therefore, $\nabla f(\gamma(s))$ is a zero eigenvector of ${\Ric}(\gamma(s))$. Hence,
\begin{align}
{\Ric}(\nabla f, Y)(\gamma(s))=0\quad\forall~Y\in T_{\gamma(s)}M.\notag
\end{align}
It follows that
\begin{align}
\langle \nabla_{\gamma^{\prime}(s)}\gamma^{\prime}(s),Y\rangle={\Ric}(\nabla f, Y)=0\quad\forall~Y\in T_{\gamma(s)}M.\notag
\end{align}
Hence, $\gamma(s)$ is a geodesic.

Let $p_i=\gamma(t_i)$ for $t_i\to+\infty$. By Lemma \ref{lem-convergence when R no decay},  $(M,g(t),p_i)$ converges subsequentially to $(M_{\infty},g_{\infty},p_{\infty})$. Moreover, there exists  a smooth function $f_{\infty}$ such that
\begin{align}
{\Ric}_{\infty}={\rm Hess} f_{\infty}\notag
\end{align}
and
\begin{align}
\nabla f\to \nabla f_{\infty},~as~i\to\infty.\notag
\end{align}
Let $\gamma_{\infty}(s)$ be the integral curve of $\nabla f_{\infty}$ passing through $p_{\infty}$ such that $\gamma_{\infty}(0)=p_{\infty}$. Similar to $\gamma(s)$, we can show that $\gamma_{\infty}(s)$ is a geodesic with respect to $g_{\infty}(0)$. Actually, we want to show that $\gamma_{\infty}(s)$ is a geodesic line. We need the following claim.
\begin{claim}\label{claim-distance function}
Suppose $a<b$ and $a,b\in \mathbb{R}$. Then,
\begin{align}
d_{\infty}(\gamma_{\infty}(a),\gamma_{\infty}(b))=\sup_{r\in\mathbb{R}}d(\gamma(t),\gamma(b-a+t)),\notag
\end{align}
where $d$ and $d_{\infty}$ are the distance functions with respect to $g$ and $g_{\infty}(0)$, respectively.
\end{claim}
\begin{proof}
By the convergence of $(M,g(t),p_i)$ and $\nabla f$, we have
\begin{align}
d_{\infty}(\gamma_{\infty}(a),\gamma_{\infty}(b))=\lim_{t_i\to+\infty}d(\gamma(t_i),\gamma(b-a+t_i)).\notag
\end{align}
So, it suffices to show that $d(\gamma(t),\gamma(b-a+t))$ is increasing in $t$. More precisely, we only need to show that
\begin{align}
d(\gamma(c),\gamma(d))\le d(\gamma(c+t),\gamma(d+t))\quad\forall~c,d\in\mathbb{R},~t>0.\notag
\end{align}
Let $l(\sigma)$ be a minimal geodesic connecting $\gamma(c)$ and $\gamma(d)$ with respect to $g(-t)$. Suppose $l(0)=\gamma(t)$ and $l(L)=\gamma(d)$. Since the Ricci curvature is nonnegative, for $t\ge 0$, we have
\begin{align*}
d_{g(-t)}(\gamma(c),\gamma(d))=&\int_{0}^{L}\sqrt{\langle l^{\prime}(s),l^{\prime}(s)\rangle_{g(-t)}}ds\\
\ge& \int_{0}^{L}\sqrt{\langle l^{\prime}(s),l^{\prime}(s)\rangle_{g}}ds\\
\ge& \; d(\gamma(c),\gamma(d)).
\end{align*}
It follows that
\begin{align}
d(\gamma(c+t),\gamma(d+t))=d_{g(-t)}(\gamma(c),\gamma(d))\ge d(\gamma(c),\gamma(d))\quad\forall~t\ge0.\notag
\end{align}
We have completed the proof of the claim.
\end{proof}

As a corollary of Claim \ref{claim-distance function}, we have
\begin{align}\label{distance identity}
d_{\infty}(\gamma_{\infty}(a),\gamma_{\infty}(b))=d_{\infty}(\gamma_{\infty}(a+t),\gamma_{\infty}(b+t))\quad\forall~a,b,t\in\mathbb{R}.
\end{align}

Now, we show that $\gamma_{\infty}(s)$ is a minimal geodesic connecting $\gamma_{\infty}(a)$ and $\gamma_{\infty}(b)$ with respect to $g_{\infty}(0)$. Suppose $l(s)$ is a minimal geodesic connecting $\gamma_{\infty}(a)$ and $\gamma_{\infty}(b)$ with respect to $g_{\infty}(0)$. Suppose $l(0)=\gamma_{\infty}(a)$ and $l(L)=\gamma_{\infty}(b)$. We want to show that
\begin{align}\label{tangent vector is eigenvector}
{\Ric}_{\infty}(l^{\prime}(s),l^{\prime}(s))=0\quad\forall~s\in[0,L].
\end{align}
For $\d>0$, we have
\begin{align}\label{distance inequality}
&d_{\infty}(\gamma_{\infty}(a),\gamma_{\infty}(b))-d_{\infty}(\gamma_{\infty}(a-\d),\gamma_{\infty}(b-\d))\notag\\
=&d_{\infty}(\gamma_{\infty}(a),\gamma_{\infty}(b))-d_{g_{\infty}(\d)}(\gamma_{\infty}(a),\gamma_{\infty}(b))\notag\\
\ge&\int_{0}^{L}\sqrt{\langle l^{\prime}(s),l^{\prime}(s)\rangle_{g_{\infty}(0)}}ds-\int_{0}^{L}\sqrt{\langle l^{\prime}(s),l^{\prime}(s)\rangle_{g_{\infty}(\d)}}ds\notag\\
=&\int_{0}^{\d}\int_{0}^{L}\frac{2{\Ric}_{g_{\infty}(\sigma)}(l^{\prime}(s),l^{\prime}(s))}{\sqrt{\langle l^{\prime}(s),l^{\prime}(s)\rangle_{g_{\infty}(\sigma)}}}dsd\sigma .
\end{align}
By (\ref{distance identity}) and (\ref{distance inequality}), we get
\begin{align}\label{integration equal to zero}
\int_{0}^{\d}\int_{0}^{L}\frac{2{\Ric}_{g_{\infty}(\sigma)}(l^{\prime}(s),l^{\prime}(s))}{\sqrt{\langle l^{\prime}(s),l^{\prime}(s)\rangle_{g_{\infty}(\sigma)}}}dsd\sigma=0.
\end{align}
Since the Ricci curvature is nonnegative, (\ref{integration equal to zero}) implies (\ref{tangent vector is eigenvector}). Note that (\ref{tangent vector is eigenvector}) implies that $l^{\prime}(s)$ is a zero eigenvector of ${\Ric}_{\infty}(l(s))$. Hence,
\begin{align}
{\Ric}_{\infty}(l^{\prime}(s), Y)(l(s))=0\quad\forall~Y\in T_{l(s)}M_{\infty}.\notag
\end{align}
Therefore,
\begin{align}
\frac{{\rm d}R_{\infty}(l(s))}{{\rm d}s}=-2{\Ric}_{\infty}(\nabla f_{\infty}(l(s)),l^{\prime}(s))=0\quad\forall~s\in[0,L].\notag
\end{align}
Hence,
\begin{align}
R_{\infty}(l(s))=R_{\infty}(\gamma_{\infty}(s))=R_{\infty}(p_{\infty}).\notag
\end{align}
By the convergence of $(M,g(t),p_i)$, we have
\begin{align}
|\nabla f_{\infty}|_{g_{\infty}(0)}(p_{\infty})=1.\notag
\end{align}
Note that
\begin{align}
|\nabla f_{\infty}|^2(x)+R_{\infty}(x)\equiv C.\notag
\end{align}
We conclude that
\begin{align}
|\nabla f_{\infty}|_{g_{\infty}(0)}(l(s))=1\quad\forall~s\in[0,L].\notag
\end{align}

For any $b>a$, we have
\begin{align}
f_{\infty}(\gamma_{\infty}(b))-f_{\infty}(\gamma_{\infty}(a))=&\int_0^L\langle l^{\prime}(s),\nabla f_{\infty}(l(s)) \rangle_{g_{\infty}(0)}\,ds\notag\\
\le& \int_0^L |l^{\prime}(s)|_{g_{\infty}(0)}\cdot|\nabla f_{\infty}(l(s))|_{g_{\infty}(0)}\,ds\notag\\
\le& L.
\end{align}
Note that $l(s)$ is a minimal geodesic. Since $\gamma_{\infty}(s)$ is a curve connecting $\gamma_{\infty}(a)$ and $\gamma_{\infty}(b)$ and the length of $\gamma_{\infty}(s)|_{[a,b]}=b-a$, we get
\begin{align}\label{9-1}
f_{\infty}(\gamma_{\infty}(b))-f_{\infty}(\gamma_{\infty}(a))\le  L\le b-a.
\end{align}
On the other hand, since  $\gamma_{\infty}(s)$ is the integral curve of $\nabla f$ and $|\nabla f|(\gamma(s))=1$, we have
\begin{align}\label{9-2}
f_{\infty}(\gamma_{\infty}(b))-f_{\infty}(\gamma_{\infty}(a))=b-a.
\end{align}
By (\ref{9-1}) and (\ref{9-2}), we have $L=b-a$.

Finally, we get $\gamma_{\infty}(s)$ is a minimal geodesic connecting $\gamma_{\infty}(a)$ and $\gamma_{\infty}(b)$ for any $a,b\in\mathbb{R}$. Hence,  $\gamma_{\infty}(s)$ is a geodesic line. Since $(M_{\infty},g_{\infty}(0))$ has nonnegative Ricci curvature, it splits off a line by the Cheeger-Gromoll splitting theorem. We have completed the proof.

\end{proof}

Now, we can prove the dimension reduction result in a more general case.

\begin{lem}\label{lem-splitting-2}
 Let $(M,g,f)$ be an $n$-dimensional  $\kappa$-noncollapsed steady gradient Ricci soliton with bounded curvature and nonnegative Ricci curvature. Suppose there exists a sequence of points $p_i$ tending to infinity such that  $R(p_i)$ attains the maximum of $R(x)$ at each $p_i$. Then, $(M,g,f)$ weakly dimension reduces to an $(n-1)$-dimensional steady gradient
Ricci soliton.
\end{lem}

\begin{proof}
If there exists a point $x_0\in M$ such that $R(x_0)$ attains the maximum of $R(x)$ and $|\nabla f|(x_0)>0$,\footnote{By rescaling, this is equivalent to the case that $|\nabla f|(x_0)=1$.} then the lemma holds according to Lemma \ref{lem-splitting-1}. Now, it suffices to consider the case that $|\nabla f|(p)=0$ if $R(p)=R_{\max}$, where $R_{\max}=\sup_{x\in M}R(x)$. Hence, $|\nabla f|(p_i)=0$ for all $i\in\mathbb{N}$.

Let $\phi_t$ be a group of diffeomorphisms generated by $-\nabla f$ and $g(t)=\phi_t^{\ast}g$. Let $\gamma_i(s)$ be a minimal geodesic connecting $p_0$ and $p_i$ with respect to $g$. Since $\phi_t$ is an isomorphism and $|\nabla f|(p_i)=0$, we have
\begin{align}
d_{g(t)}(p_i,p_0)=d(p_i,p_0)\quad\forall~t\in \mathbb{R}.\notag
\end{align}

Let $d_i=d(p_i,p_0)$. Let $L(t)$ be the length of $\gamma_i(s)|_{[0,d_i]}$ and $s$ be the arc-parameter with respect to $g(t)$. By the nonnegativity of the Ricci curvature, we have
\begin{align}
\frac{{\rm d}L(t)}{{\rm d}t}=-\int_{0}^{d_i}\frac{2{\Ric}_{g(t)}(\gamma_i^{\prime}(s),\gamma_i^{\prime}(s))}{\langle \gamma_i^{\prime}(s),\gamma_i^{\prime}(s)\rangle_{g(t)}}ds\le 0.\notag
\end{align}
Hence,
\begin{align}
L(t)\le d_i\quad\forall~t\ge0.\notag
\end{align}
On the other hand,
\begin{align}
L(t)\ge d_{g(t)}(p_0,p_i)=d_i.\notag
\end{align}
Therefore,
\begin{align}
L(t)\equiv d_i.\notag
\end{align}
It follows that
\begin{align}
{\Ric}_{g(t)}(\gamma_i^{\prime}(s),\gamma_i^{\prime}(s))\equiv0\quad\forall~s\in [0,d_i],~t\in \mathbb{R}.\notag
\end{align}
Then,
\begin{align}
{\Ric}(\gamma_i^{\prime}(s),Y)=0\quad\forall~Y\in T_{\gamma_i(s)}M.\notag
\end{align}
Hence,
\begin{align}
\frac{{\rm d}R(\gamma_i(s))}{{\rm d}s}=2{\Ric}(\nabla f(\gamma_i(s)),\gamma_i(s) )=0.\notag
\end{align}
It follows that
\begin{align}\label{identity for geodesic}
R(\gamma_i(s))=R(\gamma_i(0))=R(p_0)=R_{\max}.
\end{align}

Let $q_i=\gamma_i(\frac{d_i}{2})$. Note that $R(q_i)=R_{\max}$ and $q_i$ tends to infinity. Moreover, $\gamma_i(s)|_{[0,d_i]}$ is a minimal geodesic passing through $q_i$ and its length $d_i$ tends to infinity.  Now, we consider $(M,g(t),q_i)$. By taking a subsequence, $(M,g(t),p_i)$ converges to $(M_{\infty},g_{\infty}(t),p_{\infty})$. This means that there exist difformorphisms $\Phi_i:U_i(\subseteq M_{\infty})\to \Phi_i(U_i)(\subseteq M_i)$ such that $\Phi_i^{\ast}(g(t))$ converges to $g_{\infty}(t)$ and $\Phi_i(p_{\infty})=p_i$, where the $U_i$ exhaust $M_{\infty}$. Let
\begin{align}\label{fix direction}
V_i=\left.\frac{{\rm d}}{{\rm d}s}\right|_{s=\frac{d_i}{2}}\Phi_i^{-1}(\gamma_i(s))\quad\forall~i\in\mathbb{N}.
\end{align}
Since $s$ is the arc-parameter, we get $|V_i|_{\Phi_i^{\ast}(g)}=1$ and $V_i\in T_{p_{\infty}}M_{\infty}$. By taking a subsequence, we may assume that $V_i\to V_{\infty}$ as $i\to\infty$. Therefore, by  the convergence of $(M,g(t),q_i)$ and $V_i$, we get that $\gamma_i(s)|_{[0,d_i]}$ converges subsequentially to a geodesic line passing through $p_{\infty}$. Note that $(M_{\infty},g_{\infty}(t))$ has nonnegative Ricci curvature. We conclude that $(M_{\infty},g_{\infty}(t))$ splits off a line. Similar to the proof of Lemma \ref{lem-splitting-1}, $(M_{\infty},g_{\infty}(t))$ is also a steady gradient Ricci soliton.

\end{proof}

We also need the following lemma.

\begin{lem}\label{lem-shrinking-nonnegative}
Let $(M,g,f)$ be a non-Ricci-flat gradient steady Ricci soliton with nonnegative Ricci curvature. Let $S=\{x\in M: \nabla f(x)=0\}$. Suppose $S$ is not empty. Then, for any $p\in M$, we have
\begin{align}
 \lim_{t\to+\infty}d(\phi_t(p),S)=0,\notag
 \end{align}
 where $\phi_t$ is generated by $-\nabla f$.
\end{lem}
\begin{proof}
Let $g(t)=\phi_t^{\ast}g$. Then, $g(t)$ satisfies the Ricci flow equation. Note that $\phi_t$ is an isomorphism and the Ricci curvature of $(M,g(t))$ is nonnegative. For any $q\in M$, we have $d(\phi_t(q),\phi_t(p))=d_t(p,q)$ is decreasing in $t$ .

Suppose $d=d(p,S)$. Since $S$ is closed, we can find $q\in S$ such that $d(p,q)=d$. Hence, we have
\begin{align}
d(\phi_t(p),S)\le d(\phi_t(p),q)=d_t(p,q)\le d, ~\mbox{for}~t\ge0.\notag
\end{align}
Therefore, there exists a sequence of times $t_i\to+\infty$ such that $\phi_{t_i}(p)$ converges to some point $p_{\infty}$.

Now, we show that $p_{\infty}\in S$. If $p_{\infty}\notin S$, then we let $c_0=|\nabla f|^2(p_{\infty})>0$. Since the Ricci curvature is nonnegative, we have
 \begin{align}\label{monotonicity of norm}
\frac{{\rm d}}{{\rm d}t}|\nabla f|^2(\phi_{t}(p))=-\frac{{\rm d}}{{\rm d}t}R(\phi_{t}(p))\le 0.
\end{align}

 By the convergence of $\phi_{t_i}(p)$ and (\ref{monotonicity of norm}), we have
\begin{align}\label{limit lower bound}
|\nabla f|^2(\phi_{t_i}(p))\ge|\nabla f|^2(p_{\infty})= c_0~ \forall~ t_i\ge 0.
\end{align}
By (\ref{monotonicity of norm}) and (\ref{limit lower bound}), we have
\begin{align}
|\nabla f|^2(\phi_{t}(p))\ge c_0, ~\mbox{for}~t\in \mathbb{R}.\notag
\end{align}
Then,
\begin{align}
f(p)-f(p_{\infty})=\lim_{t_{i}\to+\infty}\int_{0}^{t_i}|\nabla f|^2(\phi_t(p)){\rm d}t\ge \lim_{t_{i}\to+\infty}c_0t_i=+\infty.\notag
\end{align}
This is impossible. Hence, we have proved that $p_{\infty}\in S$. Hence,
\begin{align}
 \lim_{t_i\to+\infty}d(\phi_{t_i}(p),p_{\infty})=0.\notag
 \end{align}
 Note that
 \begin{align}
 d(\phi_{t}(p),S)\le d(\phi_{t}(p),p_{\infty})\le d(\phi_{t_i}(p),p_{\infty})\quad\forall~t\ge t_i.\notag
 \end{align}
 Hence, we have completed the proof.
\end{proof}

Now, we prove Theorem \ref{theo-dimension reduce without decay} by assuming the Ricci curvature is nonnegative.

\begin{lem}\label{lem-nonnegative case-spliting}
Theorem \ref{theo-dimension reduce without decay} holds when the Ricci curvature is nonnegative.
\end{lem}

\begin{proof}
We first note that Lemma \ref{lem-nonnegative case-spliting} can be reduced to the case that there exists a point $p_0\in M$ such that $R(p_0)=R_{\max}$, where $R_{\max}=\sup_{x\in M}R(x)$.

If $R(x)<R_{\max}$ for any $x\in M$, then there exists a sequence of points $p_i$ tending to infinity such that $R(p_i)\to R_{\max}$. By taking a subsequence, we see that $(M,g(t),p_i)$ converges to $(M_{\infty},g_{\infty}(t),p_{\infty})$. As in the proof of Lemma \ref{lem-splitting-1}, $(M_{\infty},g_{\infty}(t),p_{\infty})$ is a gradient steady Ricci soliton with nonnegative Ricci curvature and bounded curvature. Moreover, $R_{\infty}(x,0)$ attains its maximum at $p_{\infty}$. If $(M_{\infty},g_{\infty}(0),p_{\infty})$ weakly dimension reduces to an $(n-1)$-dimensional steady Ricci soliton, then $(M,g,f)$ also weakly dimension reduces to the same steady Ricci soliton.

Hence, we may assume there exists a point $p_0\in M$ such that $R(p_0)=R_{\max}$. Let $S^{\prime}=\{x\in M: R(x)=R_{\max}\}$. If $S^{\prime}$ is unbounded, then $(M,g,f)$ weakly dimension reduces to an $(n-1)$-dimensional steady Ricci soliton by Lemma \ref{lem-splitting-2}. Therefore, we may assume that $S^{\prime}$ is a non-empty and bounded set. By Lemma \ref{lem-splitting-1}, if there exists a point $x_0\in S^{\prime}$ such that $|\nabla f|(x_0)>0$, then $(M,g,f)$ weakly dimension reduces to an $(n-1)$-dimensional steady Ricci soliton.

Finally, we only need to consider the case that $S^{\prime}$ is a non-empty and bounded set and $|\nabla f|(x)=0$, for all $x\in S^{\prime}$. Let $S=\{x\in M: \nabla f(x)=0\}$. In this case, $S^{\prime}=S$. By Lemma \ref{lem-shrinking-nonnegative}, for any $p\in M\setminus S$, we have
\begin{align}\label{9-shrinking}
\phi_t(p)\to S,~{\rm as}~t\to+\infty.
\end{align}
 By (\ref{9-shrinking}), for any $x\in M$ such that $d(x,S)\ge 1$, there exists a constant $t_x\ge 0$
 such that $d(\phi_{t_x}(x),S)=1$. Since the Ricci curvature is nonnegative, $R(\phi_t(x))$ is increasing in $t$. Hence,
 \begin{align}
 R(x)\le R(\phi_{t_x}(x))\le \sup_{d(y,S)=1}R(y)\quad\forall~d(x,S)\ge 1.\notag
 \end{align}

Let $C=\sup_{d(y,S)=1}R(y)$. Obviously, $C<R_{\max}$. Hence, we get
 \begin{align}
 R(x)\le C<R_{\max}\quad\forall~d(x,S)\ge 1.\notag
 \end{align}

 Let $A=\lim_{r\to\infty}\sup_{x\in M\setminus B(x_0,r)}R(x)$, where $x_0$ is a fixed point. Since we have assumed that the scalar curvature does not have uniform decay, we get $A>0$. We also note that $A\le C<R_{\max}$. We can choose a sequence of points $\{p_i\}$ tending to infinity such that $R(p_i)\to A$ as $i\to \infty$. It is easy to see that $(M,g(t),p_i)$ converges subsequentially to a limit $(M_{\infty},g_{\infty}(t),p_{\infty})$. Then, $(M_{\infty},g_{\infty}(t),p_{\infty})$ is a steady gradient Ricci soliton with nonnegative Ricci curvature and bounded curvature. We also have $R_{\infty}(p_{\infty},0)$ attains the maximum of $R_{\infty}(x,0)$ at the point $p_{\infty}$. By the convergence, we also have
 \begin{align}
 |\nabla f_{\infty}|^2(p_{\infty})+R_{\infty}(p_{\infty},0)=\lim_{i\to\infty}(|\nabla f|^2(p_i)+R(p_i))=R_{\max}\notag
 \end{align}
 and
 \begin{align}
 R_{\infty}(p_{\infty},0)=\lim_{i\to\infty}R(p_i)=A<R_{\max}.\notag
 \end{align}
Hence,
\begin{align}
|\nabla f_{\infty}|^2(p_{\infty})=R_{\max}-A>0.\notag
\end{align}
By Lemma \ref{lem-splitting-1}, $(M_{\infty},g_{\infty}(0),p_{\infty})$ weakly dimension reduces to an $(n-1)$-dimensional steady Ricci soliton. Therefore, $(M,g,f)$ also dimension reduces to an $(n-1)$-dimensional steady Ricci soliton. We have completed the proof.

\end{proof}

To prove Theorem \ref{theo-dimension reduce without decay}, we need to introduce the following lemma.

\begin{lem}\label{lem-geodesic}
Let $(M,g)$ be a complete Riemannian manifold and let $\{p_j\}_{j\in \mathbb{N}_{+}}$ be a sequence of points tending to infinity. Then, for any given compact set $K$,  there exist infinitely many $i,j \in \mathbb{N}_{+}$ such that any minimal geodesic connecting $p_{i}$ and $p_{j}$ is away from $K$.
\end{lem}
\begin{proof}
Let $p=p_0$. Suppose $K\subset B(p,r-1)$ for some $r>0$. Let $\gamma_j(t)$ be a minimal geodesic connecting $p$ and $p_j$, where $t$ is the arc-parameter and $\gamma_j(0)=p$. Suppose $\gamma_j(t)=\exp_p(tv_j)$ for some unit vector $v_j\in T_{p}M$. By passing to a subsequence, we may assume that $v_j\to v$ for some $v\in T_pM$. Let $\gamma(t)=\exp_p(tv)$. Fix $l>2r$. By the convergence of $v_j$, we have
\begin{align}\label{splitting lemma-3}
\gamma_j(t)\to \gamma(t)~\mbox{as}~j\to\infty\quad\forall~t\in[0,l].
\end{align}

Now, we claim that there exists a constant $j_0\in \mathbb{N}_{+}$ such that  any minimal geodesic connecting $p_{i}$ and $p_{j}$ is away from $B(p,r)$ if $i,j\ge j_0$. The lemma follows from this claim immediately.

We prove the claim by contradiction. Let $\sigma_{ij}(s)$ be a minimal geodesic connecting $p_i$ and $p_j$, where $s$ is the arc-parameter and $\sigma_{ij}(0)=p_{j}$. If the claim is not true, we may assume that $\sigma_{ij}(s_0)\in B(p,r)$ for some $s_0\in(0,d(p_i,p_j))$. Then,
\begin{align*}
d(p_i,\sigma_{ij}(s_0))\ge d(p_i,p)-d(\sigma_{ij}(s_0),p)\ge d(p_i,p)-r,\\
d(p_j,\sigma_{ij}(s_0))\ge d(p_j,p)-d(\sigma_{ij}(s_0),p)\ge d(p_j,p)-r.
\end{align*}
Therefore,
\begin{align}\label{splitting lemma-1}
d(p_i,p_j)=d(p_i,\sigma_{ij}(s_0))+d(p_j,\sigma_{ij}(s_0))\ge d(p_i,p)+d(p_j,p)-2r.
\end{align}

On the other hand, by the definition of $\gamma_i(t)$ and $\gamma_j(t)$, we have
\begin{align}\label{splitting lemma-2}
d(p_i,p_j)\le& d(p_i,\gamma_i(l))+d(\gamma_i(l),\gamma_j(l))+d(p_j,\gamma_j(l))\notag\\
=&d(\gamma_i(l),\gamma_j(l))+d(p_i,p)+d(p_j,p)-2l.
\end{align}

By (\ref{splitting lemma-1}) and (\ref{splitting lemma-2}), we get
\begin{align}\label{splitting lemma-4}
d(\gamma_i(l),\gamma_j(l))\ge 2(l-r)>2r>0.
\end{align}

However, by (\ref{splitting lemma-3})
\begin{align}
d(\gamma_i(l),\gamma_j(l))\to 0 ~\mbox{as}~i,j\to\infty.\notag
\end{align}

This contradicts (\ref{splitting lemma-4}). We have completed the proof.
\end{proof}

Now, we are ready to prove Theorem \ref{theo-dimension reduce without decay}.

\begin{proof}[Proof of Theorem \ref{theo-dimension reduce without decay}.]
Let $$R_{\max}=\sup_{x\in M}R(x),~A=\lim_{r\to\infty}\sup_{x\in M\setminus B(x_0,r)}R(x),$$ where $x_0$ is a fixed point. Since we have assumed that the scalar curvature does not have uniform decay, we get $A>0$.  We can choose a sequence of points $\{p_i\}$ tending to infinity such that $R(p_i)\to A$ as $i\to \infty$. Then, $(M,g,p_i)$ converges subsequentially to a limit $(M_{\infty},g_{\infty},p_{\infty})$, where $(M_{\infty},g_{\infty},p_{\infty})$ is a $\kappa$-noncollapsed steady gradient Ricci soliton with nonnegative Ricci curvature and bounded curvature. We also have $R_{\infty}(p_{\infty})$ attains the maximum of $R_{\infty}(x)$ at the point $p_{\infty}$.

\textbf{Case 1:}  $(M_{\infty},g_{\infty})$ does not have uniform scalar curvature decay. We apply Lemma \ref{lem-nonnegative case-spliting} to $(M_{\infty},g_{\infty},p_{\infty})$. Then,
 $(M_{\infty},g_{\infty},p_{\infty})$ weakly dimension reduces to an $(n-1)$-dimensional steady Ricci soliton. Therefore, $(M,g,f)$ also dimension reduces to an $(n-1)$-dimensional steady Ricci soliton.

\textbf{Case 2:} $(M_{\infty},g_{\infty})$ has uniform scalar curvature decay. We will exclude this case.  Note that $$R_{\infty}(p_{\infty})=A>0.$$ By the curvature decay, we can choose $r_0>0$ such that
\begin{align}
R_{\infty}(x)\le \frac{A}{2}, ~\forall~x\in ~M_{\infty}\setminus B(p_{\infty},r_0;g_{\infty}(0)).\notag
\end{align}
 By the convergence of $(M,g,p_i)$, there exists a constant $i_0>0$ such that
\begin{align}\label{inequality 1 for splitting theorem}
R(x)< \frac{3A}{4}<R(p_i)\quad\forall~i\ge i_0,~x\in \partial B(p_i,r_0;g).
\end{align}
We also assume that $K\cap B(p_i,r_0;g)=\varnothing$ for $i\ge i_0$.

Let $\phi_t$ be a group of diffeomorphisms generated by $-\nabla f$. We first claim that $\phi_t(p_i)\in B(p_i,r_0;g)$ for all $t\ge 0$ and $i\ge i_0$.

If the claim is not true, then there exists $T>0$ such that $\phi_{T}(p_i)\in \partial B(p_i,r_0;g)$ and $\phi_{t}(p_i)\in B(p_i,r_0;g)~ \forall t\in (0,T)$. Hence, $R(\phi_{T}(p_i))<R(p_i)$ by (\ref{inequality 1 for splitting theorem}). Since $\phi_{t}(p_i)\in B(p_i,r_0;g) ~\forall t\in (0,T)$ and the Ricci curvature is nonnegative on $B(p_i,r_0;g)$ for $i\ge i_0$, we know $R(\phi_t(p_i))$ is increasing for $t\in (0,T)$. Therefore, $R(\phi_{T}(p_i))\ge R(p_i)$. However, we have shown that $R(\phi_{T}(p_i))<R(p_i)$. Hence, $\phi_{t}(p_i)\in B(p_i,r_0;g)~ \forall t\ge 0$.

Next, we claim that there exists a point $q_i\in B(p_i,r_0;g)$ such that $\nabla f(q_i)=0$ for all $i\ge i_0$.

We prove the claim. Since $\phi_{t}(p_i)$ stays in  $B(p_i,r_0;g)$ and the Ricci curvature is nonnegative on $B(p_i,r_0;g)$, $R(\phi_t(p_i))$ is increasing for $t\ge0$, i.e., $|\nabla f|^2(\phi_t(p_i))=C-R(\phi_t(p_i))$ is decreasing for $t\ge0$. We assume that  $|\nabla f|^2(\phi_t(p_i))\to c$ as $t\to+\infty$. Let
$$\Delta=\sup_{x\in B(p_i,r_0;g)}f(x)-\inf_{x\in B(p_i,r_0;g)}f(x).$$
It is obvious that $\Delta>0$ is finite. Note that
\begin{align}
\Delta\ge f(p_i)-f(\phi_t(p_i))=\int_{0}^t|\nabla f|^2(\phi_s(p_i))ds\ge c^2t\quad\forall~t\ge0.\notag
\end{align}
Note that $\Delta$ is independent of $t$. By taking $t\to+\infty$, we get $c=0$. Hence, $|\nabla f|^2(\phi_t(p_i))\to 0$ as $t\to+\infty$. Since $\phi_{t}(p_i)$ stays in  $B(p_i,r_0;g)$, we may assume $\phi_{t_k}(p_i)$ converges to some point $q_i$ as $i_k\to+\infty$. Hence, $|\nabla f|(q_i)=0$ by the convergence of $\phi_{t_k}(p_i)$. We have completed the proof the claim.

By the claim, there exists a point $q_i\in B(p_i,r_0;g)$ such that $\nabla f(q_i)=0$ and $R(q_i)=R_{\max}$ for all $i\ge i_0$. By taking a subsequence, we may also assume that $B(p_i,r_0;g)\cap B(p_j,r_0;g)=\varnothing$ and $B(p_i,r_0;g)\cap K=\varnothing$ for $i,j\ge i_0$ and $i\neq j$. By Lemma \ref{lem-geodesic}, we can find $i,j\in\mathbb{N}_{+}$ such that there exists a minimal geodesic $\sigma_{ij}(s)$ connecting $q_i$ and $q_j$ such that
\begin{align}
d(\sigma_{ij}(s), K)\ge 1\quad\forall~s\in [0,d(q_i,q_j)],
\end{align}
where $s$ is the arc-parameter and $\sigma_{ij}(0)=q_i$.

Note that $\sigma_{ij}(s)$ is a minimal geodesic connecting $q_i$ and $q_j$. Moreover, $\nabla f(q_i)=\nabla f(q_j)=0$ and ${\rm Ric}(\sigma_{ij}(s))\ge0$ for all $s\in [0,d(q_i,q_j)]$. By the argument in the proof of Lemma \ref{lem-splitting-2} (see the proof of (\ref{identity for geodesic})), we get
\begin{align}\label{inequality 2 for splitting theorem}
R(\sigma_{ij}(s))=R(\sigma_{ij}(0))=R_{\max}\quad\forall~s\in [0,d(q_i,q_j)].
\end{align}

By the choice of $q_i,q_j$, we get $q_i\in B(p_i,r_0;g)$ and $q_j\notin B(p_i,r_0;g)$. Then, there exists $s_0\in (0,d(q_i,q_j))$ such that $\sigma_{ij}(s_0)\in\partial B(p_i,r_0;g)$. By (\ref{inequality 1 for splitting theorem}),
\begin{align}
R(\sigma_{ij}(s_0))<R(p_i)\le R_{\max}.\notag
\end{align}
This contradicts (\ref{inequality 2 for splitting theorem}). Hence, the scalar curvature does not have uniform decay.
\end{proof}

\section{Proofs of Theorem \ref{theo-singularity model}, Theorem \ref{theo-Kaehler-1} and Theorem \ref{theo-Kaehler-2}}\label{final section}

As we have mentioned in the introduction, Theorem \ref{theo-nonnegative Ricci} is a corollary of Theorem \ref{theorem-1} and Theorem \ref{theo-compactness}. Combining Theorem \ref{theo-nonnegative Ricci} and Theorem \ref{theo-dimension reduce without decay}, we get Theorem \ref{theo-nonnegative Ricci-bounded curvature}. Then, Theorem \ref{theo-singularity model} is a corollary of Theorem \ref{theo-nonnegative Ricci-bounded curvature}.

\begin{proof}[Proof of Theorem \ref{theo-singularity model}]
By \cite[Theorems 28.6 and 28.9]{Cetc}), $(M,g)$ is strongly $\kappa$-noncollapsed. By Theorem 1 in \cite{CFSZ}, $(M,g)$ also has bounded curvature. Hence, Theorem \ref{theo-singularity model} is true by Theorem \ref{theo-nonnegative Ricci-bounded curvature}  if  $(M,g)$ is not Ricci flat. Note that $(M,g)$ is strongly $\kappa$-noncollapsed. If $(M,g)$ ia Ricci flat, then it has maximal volume growth. By  Corollary 8.86 in \cite{CN}, $(M,g)$ must be an ALE $4$-manifold.  We have completed the proof.
\end{proof}

Next, we prove Theorem \ref{theo-Kaehler-1}.

\begin{proof}[Proof of Theorem \ref{theo-Kaehler-1}]
It suffices to exclude the case that $(M,g,f)$ is not Ricci flat. Suppose  $(M,g,f)$ is not Ricci flat. By Theorem \ref{theo-nonnegative Ricci-bounded curvature}, $(M,g_i(t),p_i)$ converges to $(N\times\mathbb{R},g_N(t)+ds^2,p_{\infty})$ for $p_i$ tending to infinity, where $g_i(t)=K_ig(K^{-1}_it)$ and $N$ are defined as in Definition \ref{def-1}. Since $(M,g)$ is a K\"{a}hler manifold, $(N\times\mathbb{R},g_N(t)+ds^2)$ is also  a K\"{a}hler manifold. Let $J_{\infty}$ be the K\"{a}hler structure of $(N\times\mathbb{R},g_N(t)+ds^2)$. Let $V$ be the parallel vector field parallel in the $\mathbb{R}$ direction, i.e., $\nabla V\equiv0$. Then, $\nabla J_{\infty}V\equiv0$. Hence, $J_{\infty}V$ is also a parallel vector field. Hence,$(N\times\mathbb{R},g_N(t)+ds^2)$ locally splits off a complex line. Since the oriented $2$-dimensional $\kappa$-solution must be a family of shrinking round spheres, $(N\times\mathbb{R},g_N(t)+ds^2)$ should be a quotient of $(\mathbb{S}^2\times\mathbb{R}^2,g_{\mathbb{S}^2}(t)+ds_1^2+ds_2^2)$. On the other hand, $N$ is either diffeomorphic to $\mathbb{S}^3/\Gamma$ or diffeomorphic to $\mathbb{R}^3$ by  Theorem \ref{theo-nonnegative Ricci-bounded curvature}. We get a contradiction. Hence, $(M,g)$ must be Ricci flat. 
This completes the proof.
\end{proof}

Finally, we prove Theorem \ref{theo-Kaehler-2}.

\begin{proof}[Proof of Theorem \ref{theo-Kaehler-2}]
By \cite[Theorems 28.6 and 28.9]{Cetc}),  Theorem 1 in \cite{CFSZ} and Theorem \ref{theo-Kaehler-1}, $(M,g)$ must be K\"{a}hler-Ricci flat. Similar to the proof of Theorem \ref{theo-singularity model}, $(M,g)$ has maximal volume growth and therefore is an ALE $4$-manifold. By Kronheimer \cite{Kr1, Kr2}, any K\"{a}hler-Ricci flat ALE of real dimension $4$ must be
hyperk\"{a}hler. Hence, we have completed the proof.
\end{proof}

\vskip30mm

\appendix
\section{Volume Comparison Theorems}

\def \R {\underline{\mathcal{R}}}

In 
this
appendix, we prove some  volume comparison theorems for 
complete
Riemannian manifolds with nonnegative Ricci curvature outside a compact set. 
These results are applicable to the setting of this paper.

Let $(M^n,g)$ be a Riemannian manifold. For 
points
$x,y\in M$, we denote by
\[
	|xy| := |x,y| := d(x,y)
\]
the distance between them. 
We write
\begin{align*}
    B_x(r)=\{y\in M:|xy|<r\}&,\quad
	\bar B_x(r)=\{y\in M:|xy|\le r\},\\
	A_x(r_1,r_2)&=B_x(r_2)\setminus \bar B_x(r_1).
\end{align*}
For any measurable set $\Omega\subset M$, we denote by $|\Omega|$ the volume of $\Omega$ induced by $g$. We write
\[
    V_x(r) := |B_x(r)|.
\]
	
We say a pointed manifold $(M^n,g,o)\in \R(\rho, \Lambda)$, 
for some constants $\rho>0$, $\Lambda\ge 0$, if $(M^n,g)$ is a complete 
noncompact
manifold 
satisfying
\[
	\Ric\ge 0\quad \text{ on } M\setminus B_{o}(\rho),
\]
and
\[
	\Ric\ge -(n-1)\Lambda/ \rho^{2}\quad \text{ on } \bar B_{o}(\rho).
\]

Clearly, if $\lambda>0,$
then
\[
	(M^n,g,o)\in \R(\rho, \Lambda)\iff
	(M^n,\lambda^2 g, o) \in \R(\rho\lambda, \Lambda).
\]

\subsection{Asymptotic volume ratio}
We shall prove that if $(M^n,g)$ has nonnegative Ricci curvature outside a compact set, we can still make sense of the notion of asymptotic volume ratio.

We first prove a variant of the Bishop--Gromov volume comparison theorem. 
\begin{prop}
\label{prop: almost mono}
Let $(M^n,g)\in \R(\rho,\Lambda).$ 
For any $x\in M,$ if $B_o(\rho)\subset B_x(a)$ for some $a>0$, then
\[
	\frac{V_x(r)-V_x(a)}{(r-a)^n}
\]
is non-increasing in $r$ for $r>a$.
\end{prop}
\begin{proof}
We write $d_x(y):=d(x,y).$ It is standard to prove that
outside of $B_{x}(a),$
\[
	\Delta d_x \le \frac{n-1}{d_x - a},
\]
in the sense of distributions. See, for example, \cite[Corollary 1.2]{SY} or \cite[Lemma 4.1]{LT87}.

Let $F(r) = {\rm vol}\big( B_x(r)\setminus B_x(a) \big).$
For $r>a,$
\begin{eqnarray*}
  F'(r)&=& \int_{\partial B_x(r)} dA\\
    &=& \frac{1}{r-a} \int_{\partial B_x(r)} (d_x-a)\frac{\partial}{\partial r} (d_x-a)dA \\
 	&=& \frac{1}{r-a}\cdot \frac{1}{2}\int_{B_x(r)\setminus B_x(a)} \Delta \big( (d_x-a)^2 \big) \\
    &=& \frac{1}{r-a} \int_{B_x(r)\setminus B_x(a)} \big( (d_x-a)\Delta (d_x-a)+|\nabla (d_x-a)|^2 \big) \\
 	&\le& \frac{n}{r-a} F(r).
\end{eqnarray*}
Hence
\[
	\frac{d}{dr} \frac{V_x(r)-V_x(a)}{(r-a)^n}
	\le \frac{nF(r)}{(r-a)^{n+1}} - \frac{nF(r)}{(r-a)^{n+1}} =0.
\]
\end{proof}

\begin{lem} \label{AVR well defined}
Let $(M^n,g,o)\in \R(\rho,\Lambda).$
Then
\[
	\AVR(g) := \lim_{r\to \infty} \frac{V_o(r)}{r^n}
\]
is well-defined and does not depend on the basepoint.
\end{lem}
\begin{proof}
By the monotonicity formula above, $\AVR(g)$ is well-defined.

For any $x,y\in M,$ put $\delta = d(x,y).$
Suppose that $B_o(\rho)\subset B_y(a)$ for some $a>0.$
For $r>0$ sufficiently large,
\[
	r^{-n}V_x(r) \le r^{-n}V_y(r+\delta) \le
	r^{-n} \big( V_y(r)-V_y(a) \big) \frac{(r+\delta - a)^n}{(r- a)^n} + r^{-n}V_y(a).
\]
Hence
\[
	\lim_{r\to \infty} \frac{V_x(r)}{r^n} \le \lim_{r\to \infty} \frac{V_y(r)}{r^n}.
\]
By the symmetry of the roles of $x,y,$ $\AVR(g)$ does not depend on the basepoint.
\end{proof}

\subsection{Volume comparison for small radii}

In 
a
previous version of our paper, we tried to apply Theorem 1 of Mahaman \cite{Ma}. However, there is a gap in the proof of Theorem 1 therein, and 
in fact 
we provide a counterexample 
below.
We are very grateful for the anonymous referee for pointing out a gap in the proof in \cite{Ma}.

In \cite{Ma}, 
Bazanfar\'{e}'s proof actually yields
the following volume comparison theorem.

\begin{theo}
Suppose that $(M^n,g,o)\in \R(\rho,\Lambda).$  
For any $x\in M\setminus B_{\rho}(o),$
	$\frac{V_x(r)}{r^n}$
is decreasing for $r<\ell-\rho, $ where $\ell:=|ox|.$
For $r\ge \ell$, $0<s<r,$
we have
\[
	\frac{V_x(r)}{V_x(s)}
	\le C(\Lambda,n)\left(1+\tfrac{\Lambda\ell}{\rho}\right)^{n-1} 
	\left(\frac{r}{s}\right)^{n}.
\]
\end{theo} 


This result was also previously asserted by Cai in \cite{Cai91}.
The constant in this comparison theorem depends on the distance to the origin and it is almost optimal as 
is
seen from the example below.
\begin{proof}[Example]
Let $N$ be a 
Riemannian 
manifold with $\Rm>0,{\rm AVR}>0.$ 
Let $M=N\# (\mathbb{S}_{\epsilon}^{n-1}\times [0,\infty))$ be the connected sum 
of $N$
with 
a thin cylinder of radius $\epsilon$. 
Pick a point $x$ on the cylinder with $\ell=|ox|\gg 10.$ Then
\[
	\frac{V_x(2\ell)}{V_x(\ell/2)} \sim 
	\frac{\epsilon^{n-1}\ell + \ell^n}{\epsilon^{n-1}\ell}
	\sim (\ell/\epsilon)^{n-1}.
\]
\phantom\qedhere
\end{proof}

This example has two ends. We may wonder if 
there is
a better volume comparison theorem
just
assuming 
in addition
that
the manifold is connected at infinity.
However, we may need the following stronger topological condition, possibly because the topology of a smooth manifold is not 
as
rigid under Ricci curvature restrictions.

We say 
that 
a pointed Riemannian manifold $(M^n,g,o)$ has \textbf{connected annuli at distances at least} $r_0>0$, if for any $r\ge r_0, $ there is an open set
$\Omega_r$ such that
\begin{equation}
\label{cond: connected annuli}
\tag{{\rm CA}}
\Omega_r \text{ is connected, }\quad
\text{ and }
A_o(r/2,2r) \subset \Omega_r \subset A_o(r/3,3r).
\end{equation}



\begin{theo}
\label{thm: annuli vol cmp}
 Suppose that $(M^n,g,o)\in \R(\rho,\Lambda)$ and satisfies \eqref{cond: connected annuli} at distances at least $r_0\ge 10(1+\rho).$
Then for any $r\ge r_0$, $x\in \partial B_o(r)$, $\alpha\in (0,1/10],$
\[
	\frac{|A_o(r/2,2r)|}{|B_x(\alpha r)|} \le C(n,\Lambda,\alpha).
\] 
\end{theo}

\begin{proof}
For any $r\ge r_0$, let $\Omega_r$ be the connected domain as in \eqref{cond: connected annuli}.
By Theorem 1 in Zhong-Dong Liu's thesis \cite{Liu91} (which is available online),
we can pick a  set of points $\{p_1,\ldots,p_N\}$ in $\Omega_r$ satisfying
\[
	\Omega_r\subset \bigcup_{i=1}^N B_{p_i}(\alpha r),
\]
where
\[
	N\le \bar N =\bar N(n,\Lambda,\alpha).
\]
\textbf{Claim:} For any $1\le a,b\le N,$ we can find a subsequence $i_0,i_1,\ldots, i_m$ such that
\[
	a=i_0,\quad b=i_m,\quad
	|p_{i_{j}} p_{i_{j+1}}| \le 2\alpha r,\quad
	m\le \bar N.
\]
\begin{proof}[Proof of Claim]
This is essentially because $\Omega_r$ is connected. c.f. \cite[Corollary 2 on p.~21]{Liu91}.
Let $W_0 = B_{p_a}(\alpha r).$ Define $W_k$ to be the union of $W_{k-1}$ with those balls $B_{p_i}(\alpha r)$ 
satisfying 
$W_{k-1}\cap B_{p_i}(\alpha r)\neq \emptyset.$ This process stops in at most $N\le \bar N$ steps. If there is any ball $B_{p_j}(\alpha r)$ that does not intersect $W_N$, then we can find two open sets in $\Omega_r$ that do not intersect. This is a contradiction to the fact that $\Omega_r$ is connected.
\end{proof}

If $y,z\in \Omega_r,$ $|yz|\le 2\alpha r,$ then
\[
	V_y(\alpha r) \le V_z(3\alpha r) \le 3^nV_z(\alpha r),
\]
since $B_z(3\alpha r)\subset M\setminus B_o(\rho)$ if $\alpha\le 1/10.$
So for any indices $1\le a,b\le N,$
\[
	V_{p_a}(\alpha r)\le 3^{n\bar N} V_{p_b}(\alpha r).
\]
For any $x\in \partial B_o(r),$ there exists $b\le N$ such that $x\in B_{p_b}(\alpha r).$
Then $B_{p_b}(\alpha r)\subset B_x(2\alpha r)$, and
\[
	V_x(\alpha r)\ge 2^{-n} V_x(2\alpha r)
	\ge 2^{-n}V_{p_b}(\alpha r).
\]
It follows that
\[
	\frac{|A_o(r/2,2r)|}{|B_x(\alpha r)|}\le 	\frac{|\Omega_r|}{|B_x(\alpha r)|}
	\le 2^n\sum_{a=1}^N \frac{V_{p_a}(\alpha r)}{V_{p_b}(\alpha r)}
	\le 2^n 3^{n\bar N}\bar N .
\]
\end{proof}

\begin{cor}\label{cor-uniform volume lower bound}
 Suppose that $(M^n,g,o)\in \R(\rho,\Lambda)$ and satisfies \eqref{cond: connected annuli} at distances at least $r_0\ge 10(1+\rho).$ 
 If $\AVR(g)>0,$ then for any $x\notin B_o(r_0)$ and any $r>0,$
 \[
    \frac{V_x(r)}{r^n} \ge c(n,\Lambda) {\rm AVR}(g).
 \]
\end{cor}
\begin{proof}
For any $x\in M\setminus B_o(r_0)$, let $\ell=|ox|$. Then, $x\in \partial B_o(\ell)$ and $\ell\ge r_0$.  By Theorem \ref{thm: annuli vol cmp}, 
\begin{align}\label{appendix-1}
    \frac{V_o(2\ell)-V_o(\ell/2)}{V_x(\ell/10)}
    \le C(n,\Lambda).
\end{align}
Note that the Ricci curvature is nonnegative on $B_x(\frac{\ell}{10})$.  By (\ref{appendix-1}), for any $r\in (0,\ell/10],$  we have 
\begin{align*}
    \frac{V_x(r)}{r^n}
    &\ge \frac{V_x(\ell/10)}{(\ell/10)^n}
    \ge \frac{V_o(2\ell)-V_o(\ell/2)}{C(n,\Lambda)(2\ell - \ell/2)^n}.
\end{align*}
For any $s\ge2\ell$, by the monotonicity formula Proposition \ref{prop: almost mono}, it follows that
\begin{align*}
    \frac{V_x(r)}{r^n}
    \ge \frac{V_o(2\ell)-V_o(\ell/2)}{C(n,\Lambda)(2\ell - \ell/2)^n}
    \ge \frac{V_o(s)-V_o(\ell/2)}{C(n,\Lambda)(s - \ell/2)^n}.
\end{align*}
By taking $s\to \infty$, for $r\in (0,\ell/10],$
\[
    \frac{V_x(r)}{r^n} \ge c(n,\Lambda)\AVR(g).
\]
For any $r\in (\ell/10,4\ell],$ 
\[
    \frac{V_x(r)}{r^n}
    \ge \frac{V_x(\ell/10)}{(4\ell)^n}
    \ge \frac{c(n,\Lambda)}{40^n}\AVR(g).
\]
Note that $B_o(\rho)\subset B_x(2\ell).$
For any $r\ge 4\ell,$  by Lemma \ref{AVR well defined} and Proposition \ref{prop: almost mono}, we have
\[
    \frac{V_x(r)}{r^n}
    \ge \frac{V_x(r)-V_x(2\ell)}{(r-2\ell)^n} \frac{(r-2\ell)^n}{r^n}
    \ge 2^{-n}\frac{V_x(r)-V_x(2\ell)}{(r-2\ell)^n}
    \ge 2^{-n} \AVR(g).
\]
    
\end{proof}


Now we give a criterion for \eqref{cond: connected annuli} that will be useful in our setting.

\begin{lem}
\label{lem: CA criterion}
Suppose that $(M^n,g)$ is a complete 
Riemannian
manifold and $M$ is connected at infinity. 
Suppose that there is a smooth positive function $\beta$ on $M$ satisfying
\[
    \lim_{x\to \infty} \frac{\beta(x)}{|ox|} = 1,
\]
for some point $o\in M$.
Suppose that there is $r_0>0$ such that $|\nabla \beta|(x)>0$ for any $x\in \{\beta(x)\ge r_0\}$. 
Then $(M^n,g,o)$ satisfies \eqref{cond: connected annuli}. 
\end{lem}
\begin{proof}
By assumption, $M=\{\beta> r_0\}\cup \{\beta\le r_0\}$. Note that 
$\{\beta\le r_0\}$ is compact and  $\{\beta> r_0\}$ is diffeomrphic to $\{\beta=r_0\}\times (0,+\infty)$. Since $M$ is connected at infinity,  $\{\beta=r_0\}$ must be connected.
For 
sufficiently large $r$, we may choose
\[
    \Omega_r := 
    \left\{x: \tfrac{5r}{12}
    < \beta(x) < \tfrac{5r}{2}\right\}.
\]
Hence, $\Omega_r$ is connected.
\end{proof}

\begin{cor}
\label{uniform vol ratio}
Suppose that $(M^n,g,o)\in \R(\rho,\Lambda)$. Suppose that either: 
\begin{itemize}
    \item $\sec\ge 0$ outside a compact set and $M$ is connected at infinity; or
    \item there is a smooth function $f$ on $M$ such that $(M,g,f)$ is a steady gradient Ricci soliton and the scalar curvature decays uniformly.
\end{itemize}
Then $(M^n,g,o)$ satisfies condition \eqref{cond: connected annuli} at distances at least $r_0$ for some large $r_0>0$.
As a consequence of Corollary \ref{cor-uniform volume lower bound} , if in addition $\AVR(g)>0,$ then for $x\notin B_o(r_0),$ and any $r>0,$
\[
    \frac{V_x(r)}{r^n}\ge c(n,\Lambda)\AVR(g).
\]
\end{cor}
\begin{proof}
 For $(M^n,g)$ with sectional curvature nonnegative outside a compact set, Li and Tam proved $(M^n,g)$ satisfies \eqref{cond: connected annuli} when $M$ is connected at infinity (see  \cite[Section 2]{LT87}).

If $(M,g,f)$ is a steady soliton with uniformly decaying scalar curvature, by Theorem \ref{theo-linear of f} and Remark \ref{rem: linear f}, $f$ has 
linear growth and 
the
level sets of $f$ are in fact diffeomorphic to each other outside a compact set. Here, we normalized the steady soliton 
so
that $R+|\nabla f|^2 = 1.$
It was proved by Munteanu and Wang in \cite[Corollary 1.1]{MW} that $M$ is connected at infinity.  Hence the conditions in Lemma \ref{lem: CA criterion} are satisfied.
\end{proof}

\end{document}